\documentclass[11pt,reqno, oneside]{amsart}
\usepackage[hidelinks]{hyperref}
\usepackage{caption}
\usepackage{amssymb,amsfonts,amscd,amsbsy, color, amsmath}
\usepackage[mathscr]{eucal}
\usepackage{url}
\usepackage{graphicx}
\usepackage{mathtools}
\usepackage[textwidth=25mm,textsize=tiny]{todonotes}
\usepackage{tabularx}
\usepackage{tabu}
\usepackage{float, color}
\usepackage{placeins}
\usepackage{enumitem}

\usepackage{verbatim}
\usepackage{soul}

\usepackage{tikz}
\usetikzlibrary{math}

\usetikzlibrary{shapes,snakes}
\usetikzlibrary{plotmarks}
\usetikzlibrary{knots}

\usepackage{cite}

\newtheorem{theorem}{Theorem}[section]
\newtheorem{lemma}[theorem]{Lemma}
\newtheorem{proposition}[theorem]{Proposition}
\newtheorem{corollary}[theorem]{Corollary}

\theoremstyle{definition}
\newtheorem{definition}[theorem]{Definition}
\newtheorem{remark}[theorem]{Remark}
\numberwithin{equation}{section}
\newtheorem{example}[theorem]{Example}
\newtheorem{question}[theorem]{Question}

\newcommand{\Z}{\mathbb{Z}}
\def\sign{\operatorname{sign}}

\makeatletter
\def\imod#1{\allowbreak\mkern5mu({\operator@font mod}\,\,#1)}
\makeatother
\allowdisplaybreaks

\usepackage{hyperref}
\hypersetup{
  colorlinks,
  linkcolor={blue},
  citecolor={cyan},
  urlcolor={blue}
}

\definecolor{sagegreen}{rgb}{0.0, 0.5, 0.0}

\renewcommand{\=}{\;=\;}

\newcommand{\calT}{\mathcal T}
\newcommand{\calV}{\mathcal V}
\newcommand{\calE}{\mathcal E}

\newcommand{\Q}{\mathbb{Q}}
\newcommand{\C}{\mathbb{C}}

\def\thin{\hspace{.5pt}}
\newcommand{\Li}{\operatorname{Li}}
\renewcommand{\log}{\operatorname{Log}}

\newcommand{\bea}{\begin{equation}\begin{aligned}}
\newcommand{\eea}{\end{aligned}\end{equation}}
\newcommand{\bean}{\begin{equation*}\begin{aligned}}
\newcommand{\eean}{\end{aligned}\end{equation*}}

\newcommand{\baaa}{\begin{align}}
\newcommand{\eaaa}{\end{align}}
\newcommand{\qbinom}{\genfrac{[}{]}{0pt}{}}

\makeatletter
\@namedef{subjclassname@2020}{%
  \textup{2020} Mathematics Subject Classification}
\makeatother

\begin{document}

\title[Arborescent links and modular tails]{Arborescent links and modular tails}

\author{Robert Osburn}
\author{Matthias Storzer}

\address{School of Mathematical Sciences, University College Cork, Cork, Ireland}

\email{robert.osburn@ucc.ie}
\email{mstorzer@ucc.ie}


\date{\today}

\begin{abstract}
We prove an explicit formula for the tail of the colored Jones polynomial for a class of arborescent links in terms of a product of theta functions and/or false theta functions. We also provide numerical evidence towards a conjectural classification of the modularity of tails of the colored Jones polynomial for alternating knots.
\end{abstract}

\maketitle

\section{Introduction}
Let $K$ be a knot and $J_N(K;q)$ be the $N$th colored Jones polynomial, normalized to be 1 for the unknot and the classical Jones polynomial for $N=2$. The colored Jones polynomial features prominently in many open problems in quantum topology. For example, a central theme in this area is the Volume Conjecture  \cite{g, kashaev, hm, murakami-murakami, my}. This conjecture relates the asymptotic behavior of $J_N(K;q)$ evaluated at an $N$th root of unity $\zeta_N := e^{\frac{2\pi i}{N}}$ to the simplicial volume (or Gromov norm) of the knot complement. Precisely, 
\begin{equation} \label{vc}
\lim_{N \to \infty} \frac{\log | J_{N}(K; \zeta_N)|}{N} \= \frac{\text{Vol}(S^3 \setminus K)}{2 \pi}
\end{equation}
for all knots $K$. Here and throughout, we write $\log$ for the principal branch of the complex logarithm. One consequence of this conjecture is the following: if (\ref{vc}) is true, then a knot $K$ is trivial if and only if all of its colored Jones polynomials $J_{N}(K;q)$ are trivial. The conjecture has been proven for the following knots and links \cite{ty}: torus knots, all hyperbolic knots with at most seven crossings, Borromean rings, twisted Whitehead links, Whitehead chains and, very recently, hyperbolic double twist knots \cite{jm}. In \cite{z}, Zagier formulated a generalization of (\ref{vc}) which also incorporated the quantum modularity of $J_{N}(K;q)$ for hyperbolic knots $K$. For recent work in this direction, see \cite{bd, gz1, gz2, wh}.

In this paper, we are interested in stability properties for the coefficients of $J_{N}(L;q)$ where $L$ is a link. The {\it tail} of the colored Jones polynomial of a link $L$ (if it exists) is a power series $\Phi_{L}(q)$ whose first $N$ coefficients agree (up to a common sign) with the first $N$ coefficients of $J_N(L;q)$ for all $N \geq 1$. In 2006, Dasbach and Lin \cite{dl} conjectured that the tail exists for all alternating links $L$. This conjecture was first resolved by Armond \cite{a} and then another proof and an explicit $q$-multisum expression for the tail of $J_N(L;q)$ were given in \cite{glz}. Subsequent intriguing developments include proving the existence (and non-existence) of the tail for families of knots and links \cite{ad1, ad2, ck, eh1, eh2, ehs, ehl, h1, h2, h3, lee2, lv}, the categorification of the tail \cite{lee1, r, w}, higher order stability \cite{b, hall}, higher rank tails \cite{kk, y1, y2} and connections to representation theory and vertex operator algebras \cite{hs, sk1, sk2, sk3}. 

One can also find in \cite{glz} a table of forty-three conjectural identities between $\Phi_{K}(q)$ and products of theta functions and/or false theta functions, namely for a positive integer $b$, define
\begin{equation} \label{hbdef}
h_{b}\=
h_{b}(q) \=
\sum_{n \in \mathbb{Z}} \epsilon_{b}(n) q^{\frac{bn(n+1)}{2} - n}
\end{equation}
\noindent where
\begin{equation} \label{epsilon}
 \epsilon_{b}(n) \= \left\{
  \begin{array}{ll}
    (-1)^n & \text{if $b$ is odd,}\\
    1 & \text{if $b$ is even and $n \geq 0$,} \\
    -1 & \text{if $b$ is even and $n < 0$.} 
  \end{array} \right. 
\end{equation}
This table \cite[Table 6]{glz} consists of all alternating knots up to $8_4$, the twist knots $K_p$, $p>0$ or~$p<0$, the torus knots $T(2,p)$, $p>0$, each of their mirror knots $K^{*}$ and $8_5^{*}$. For example, we have
\begin{equation} \label{52c}
\begin{aligned}
\Phi_{5_2}(q) &\= (q)^{5}_{\infty} \sum_{a,b,c,d,e \geq 0} \frac{q^{2a^2 + ac + ad + ae + b^2 + be + cd + de + a + c+ d +e}}{(q)_a (q)_{a+c} (q)_{a+d} (q)_{a+e} (q)_b (q)_{b+e} (q)_c (q)_d (q)_e} \\
& \stackrel{?}{\=} h_{4}
\end{aligned}
\end{equation}
\noindent where we use the standard $q$-Pochhammer symbol
\begin{equation*}
(a)_n = (a;q)_n \coloneqq \prod_{k=1}^{n} (1-aq^{k-1}),
\end{equation*}
valid for $n \in \mathbb{N} \cup \{ \infty \}$. Note that $h_1=0$, $h_2=1$ and $h_3=(q)_{\infty}$. In general, $h_{b}$ is a theta function if $b$ is odd and a false theta function if $b$ is even. The modularity in the former situation is classical \cite{bvz} while false theta functions are only recently known as examples of quantum modular forms \cite[Section 4.4]{go}. Andrews \cite{an} verified the conjecture for the knots $3_1$, $4_1$ and $6_2$. In \cite{ko}, Keilthy and the first author proved not only (\ref{52c}), but {\it all} of the remaining conjectural identities in \cite{glz} via a unified $q$-theoretic approach. This approach was then used in \cite{bo} to extend \cite[Table 6]{glz} to all alternating knots up to ten crossings. Curiously, there are entries in \cite[Tables 1 and 2]{bo} and \cite[Table 6]{glz} in which a conjectural identity for $\Phi_{K}(q)$ is not known. Moreover, there is no known conjectural identity for {\it any} alternating knot (or its mirror) from $10_{79}$ to $10_{123}$.

Our main goal in this paper is to prove an explicit formula for the tail of the colored Jones polynomial for a class of alternating links in terms of products of $h_b$'s. In order to state our main result, we require some setup. For further details, see Section 2. A weighted tree\footnote{We assume that all weighted trees are reduced. See Remark \ref{reduced} (ii).} $\Gamma =  (\mathcal{V}, \mathcal{E}, w)$ is a finite planar tree with vertex set $\mathcal{V}$, edge set $\mathcal{E}$ and weight $w(v) \in \mathbb{Z}$ associated to one section around $v \in \mathcal{V}$. Given $\Gamma$, one can associate a link $L$ henceforth called an {\it arborescent link}. $\Gamma$ is called {\it alternating} if there exists a bipartition $\mathcal{V}_{+} \cup \mathcal{V}_{-}$ of $\mathcal{V}$ with $w(\calV_{+}) \subset \mathbb{Z}_{\geq 0}$ and $w(\calV_{-}) \subset \mathbb{Z}_{\leq 0}$. Here, $w(\mathcal{V}_{+})$ (respectively, $w(\mathcal{V}_{-})$) is the set of weights in $\mathcal{V}_{+}$ (respectively, $\mathcal{V}_{-}$). Note that $\Gamma$ is alternating if we can choose a sign for each vertex of weight $0$ such that the sets $\calV_\pm$ of vertices whose weights have sign $\pm $ form a bipartition of $\calV$. If $\Gamma$ is alternating, then so is the corresponding link $L$. Our main result is now the following.

\begin{theorem} \label{main} Let $\Gamma = (\mathcal{V}, \mathcal{E}, w)$ be an alternating weighted tree with arborescent link $L$. If $0 \not \in w(\mathcal{V}_{-})$, then
\begin{equation} \label{form}
\Phi_{L}(q) \= \prod_{v \in \mathcal{V}_{+}} h_{w(v) + e(v)}
\end{equation}
where $e(v)$ is the number of edges adjacent to $v$.
\end{theorem}

\begin{remark} \label{rem1}
Given a weighted tree $\Gamma$ with arborescent link $L$, the mirror image $L^{*}$ is constructed from the weighted tree obtained by flipping the signs of the weights of $\Gamma$. Thus, by Theorem \ref{main}, if $0 \not\in w(\mathcal{V}_{+})$, we have
\begin{equation} \label{mform}
\Phi_{L^{*}}(q) \= \prod_{v \in \mathcal{V}_{-}} h_{-w(v) + e(v)}.
\end{equation} 
\end{remark}

\begin{remark}
By Theorem \ref{main} and Remark \ref{rem1}, one recovers {\it all} of the entries in \cite[Tables 1 and 2]{bo} and \cite[Table 6]{glz} without a ``?". For example, given the weighted tree
\begin{equation*}
\tikz{
\draw[thick] (1,0) -- (2,0);
\filldraw (1,0) circle (2pt)
 node[anchor=south] {$-2$};
\filldraw (2,0) circle (2pt)
 node[anchor=south] {$3$};
} 
\end{equation*}
we will see in Section 2 that the associated knot is $K=5_2$. By (\ref{form}) and (\ref{mform}), we have
\begin{equation} \label{eq:tail52}
\Phi_{5_2}(q) \= h_4,\qquad\quad \Phi_{5_2^{*}}(q) \= h_3.
\end{equation}
We note that \cite{bo, ko} and \cite{glz} use different conventions for both the naming of knots and the computation of $\Phi_K$. This explains the agreement with our formulas for some knots (e.g., (\ref{eq:tail52})) while other cases are switched (e.g., (\ref{eq:tail74}) and (\ref{tail916})). For the remaining {arborescent knots $K$} with ``question marks", the obstruction to finding an identity and thus determining the type of modularity for $\Phi_{K}(q)$ is the failure of $0 \not\in w(\mathcal{V}_{-})$ in Theorem~\ref{main}. This failure is equivalent to the fact that the reduced Tait graph of $K$ is {\it not} the edge-connected sum of polygons (see Corollary~\ref{cor:0inwVTait}). We discuss these cases further in Section~5. 
\end{remark}

The paper is organized as follows. In Section 2, we recall the construction of arborescent links as given in, e.g., \cite[Chapter 17]{afhkln}, \cite[Chapter 12]{bs} or \cite[Chapter 1]{g}. In Section 3, we prove Theorem \ref{main} using properties of reduced Tait graphs for arborescent links (in particular, see the key result Proposition~\ref{prp:taitgarph_arbo_zeroless}). In Section 4, we give applications of Theorem \ref{main} to various examples of arborescent knots. In Section 5, we discuss asymptotic properties of $\Phi_{K}(q)$ when Theorem~\ref{main} is not applicable. These asymptotics suggest that the knots in Theorem~\ref{main} are the {\it only} class of alternating knots $K$ such that $\Phi_{K}(q)$ is a product of $h_b$'s and hint towards a potential connection to a complexified version of~\eqref{vc} (see Question~\ref{q} and Remarks~\ref{rq} and~\ref{rvc}).

\section{Preliminaries}
We begin by recalling some background from knot theory.

\subsection{Tait graphs}\label{sec:tait_graph}
Let $L$ be an alternating link and $D$ its associated {\it sign-colored} diagram which satisfies the rule in Figure \ref{fig:cross} at each crossing of $L$. The $\pm$-\emph{Tait graphs} $\mathcal T_\pm$ of an alternating link with sign-colored diagram $D$ are the graphs with vertices corresponding to the $\pm$-colored faces of $D$. Two vertices form an edge if the corresponding faces share a crossing. For a given alternating link, $\mathcal{T}_{+}$ is dual to $\mathcal{T}_{-}$. Moreover, the $+$-Tait graph for $L$ is the $-$-Tait graph for the mirror $L^*$ and vice versa. The {\it reduced} Tait graphs $\mathcal{T}_{\pm}^{\prime}$ are obtained from $\mathcal{T}_{\pm}$ by replacing every set of two edges that connect the same two vertices by a single edge and removing loops.

\begin{figure}[H]
\centering
\begin{tikzpicture}
\begin{knot}[clip width = 8,flip crossing=1]
\strand[very thick]
    (-1,-1) to[out=45,in=180+45] (1,1); 
\strand[very thick]
    (-1,1) to[out=-45,in=180-45] (1,-1); 
\end{knot}
\node at (3/4,0){$+$};
\node at (-3/4,0){$+$};
\node at (0,3/4){$-$};
\node at (0,-3/4){$-$};
\end{tikzpicture}
\caption{$+$ and $-$ regions.}
\label{fig:cross}
\end{figure}
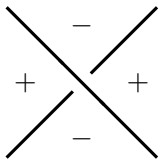
\noindent For example, a sign-colored diagram, Tait graphs and reduced Tait graphs for $K=5_2$ are given in Figure \ref{fig:dTrT}.

\begin{figure}[ht]
\begin{equation*}
\begin{tikzpicture}[scale=.6]
  \begin{knot}[
    clip width = 4,
    flip crossing/.list={1,3,5}
    ]
\strand[very thick,black]
    (5.75,1.25+.25)
    to[out=-90+45,in=90] (6.75,.25+.25)
    to[out=-90,in=90] (5.75,-1+.25+.25)
    to[out=-90,in=-180+45] (8,-.75+.25)
    to[out=45,in=-180] (9,0.25+.25)
    to[out=0,in=90] (10.5,-1.5+.25);
\strand[very thick, looseness=1,black]
    (10.5,-1.5+.25)
    to[out=-90,in=-90-45] (5.75,-2+.25+.25)
    to[out=45,in=-90] (6.75,-1+.25+.25)
    to[out=90,in=-90] (5.75,.25+.25)
    to[out=90,in=90+45+0] (8,0.25+.25)
    to[out=-45+0,in=180] (9,-.75+.25)
    to[out=0,in=180+45+45] (10.5,1.25+.25)
    to[out=90,in=90+45] (5.75,1.25+.25);
\end{knot}
\node[gray] at (9.25,0){$-$};
\node[gray] at (7.25,0){$-$};
\node[gray] at (8,1.5){$+$};
\node[gray] at (8,-1.5){$+$};
\node[gray] at (5,0){$-$};
\node[gray] at (6.25,.5){$+$};
\node[gray] at (6.25,-.5){$+$};
\end{tikzpicture}
\end{equation*}

\begin{equation*}\begin{aligned}
%%%%%%%%%%%%%%%%%%%%
%%% for + Tait graph
\tikzmath{\shift = 6;}
\begin{tikzpicture}[scale=1.2]
\node at (-2,0){$\calT_+\ $};
  \node (v1) at (1/2,1){};
  \node (v2) at (-1/2,1/2){};
  \node (v3) at (-1/2,-1/2){};
  \node (v4) at (1/2,-1){};
\filldraw (v1) circle (2pt);
\filldraw (v2) circle (2pt);
\filldraw (v3) circle (2pt);
\filldraw (v4) circle (2pt);

\draw[thick] (v1.center) -- (v2.center);
\draw[thick] (v2.center) -- (v3.center);
\draw[thick] (v3.center) -- (v4.center);
\draw[thick] (v1.center) to[out=-45,in=45] (v4.center);
\draw[thick] (v1.center) to[out=-90-45,in=90+45] (v4.center);
\node at (-2+\shift,0){$\calT_+'\ $};
  \node (v1) at (1/2+\shift,1){};
  \node (v2) at (-1/2+\shift,1/2){};
  \node (v3) at (-1/2+\shift,-1/2){};
  \node (ghost) at (1+\shift,0){};
  \node (v4) at (1/2+\shift,-1){};
\filldraw (v1) circle (2pt);
\filldraw (v2) circle (2pt);
\filldraw (v3) circle (2pt);
\filldraw (v4) circle (2pt);

\draw[thick] (v1.center) -- (v2.center);
\draw[thick] (v2.center) -- (v3.center);
\draw[thick] (v3.center) -- (v4.center);
\draw[thick] (v1.center) to (v4.center);
\end{tikzpicture}
\\[15pt]
\tikzmath{\shift = 6;} 
\begin{tikzpicture}[scale=1.2]
\node at (-2,.375){$\calT_-\ $};
\node (v1) at (-1,0){};
\node (v2) at (1,0){};
\node (v3) at (0,1.5){};
\filldraw (v1) circle (2pt);
\filldraw (v2) circle (2pt);
\filldraw (v3) circle (2pt);
\draw[thick] (v1.center) -- (v2.center);
\draw[thick] (v2.center) -- (v3.center);
\draw[thick] (v1.center) -- (v3.center);
\draw[thick] (v1.center) to[out=45,in=180] (0,1/2) to[out=0,in=90+45] (v2.center);
\draw[thick] (v1.center) to[out=-45,in=180] (0,-1/2) to[out=0,in=-90-45] (v2.center);
\node at (-2+\shift,.375){$\calT_-'\ $};
\node (v1) at (-1+\shift,0){};
\node (v2) at (1+\shift,0){};
\node (v3) at (0+\shift,1.5){};
\filldraw (v1) circle (2pt);
\filldraw (v2) circle (2pt);
\filldraw (v3) circle (2pt);

\draw[thick] (v1.center) -- (v2.center);
\draw[thick] (v2.center) -- (v3.center);
\draw[thick] (v1.center) -- (v3.center);
\end{tikzpicture}                                                                                                                                                                                                                                                                                                                                                                                                                                                                                                                                                                                                                                                                                                                                                                                                                                                                                                                    
\end{aligned}
\end{equation*}
\caption{Sign-colored diagram, $\mathcal{T}_{\pm}$ and $\mathcal{T}_{\pm}^{\prime}$ for $K=5_2$.}
\label{fig:dTrT}
\end{figure}
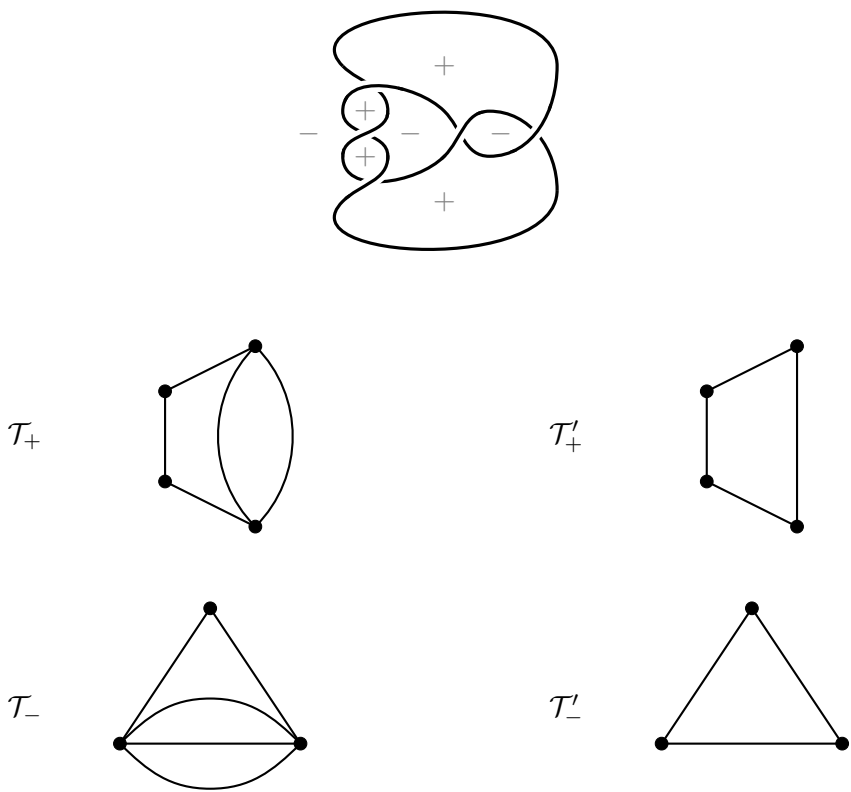

\subsection{Arborescent links} \label{ak}
Following \cite[Chapter 17]{afhkln}, \cite[Chapter 12]{bs} or \cite[Chapter 1]{g}, we introduce arborescent links which are a class of links associated to weighted trees. The construction of arborescent links described below is equivalent to that of Conway’s algebraic links \cite{c}, see \cite[Chapter 14.3]{bs} for a direct comparison.

\begin{definition}
A \emph{weighted tree} $\Gamma = (\calV,\calE,w)$ is a planar embedding of a finite tree $(\calV,\calE)$ together with a weight $w(v) \in \mathbb{Z}$ assigned to one section around $v$ for each vertex $v\in\calV$.
\end{definition}
For a given weighted tree $\Gamma$, the weights are depicted as integers written in its sections. We now proceed as follows: 

\begin{enumerate}
\item Let $v\in\calV$ be a vertex with weight $w(v)$ and $n\in\Z_{\geq 0}$ adjacent vertices $v_1,\ldots, v_n\in\calV$ in counterclockwise order around $v$. We construct a ribbon associated to $v$ that has $n$ marked squares corresponding to $v_i$, $i=1,\ldots,n$, followed by $w(v)$ half-twists, see Figure~\ref{fig:ribbonv}. We use the convention that
\tikz[baseline=-2ex,scale=.7]{
\begin{knot}[clip width = 4]
% v1
\strand[very thick] (0,0) to[out=00,in=180] (1,-.5); 
\strand[very thick] (0,-.5) to[out=00,in=180] (1,0); 
\end{knot}
}
is a positive half-twist. The ribbon has two orientations: a horizontal core orientation ($c$) and a vertical normal orientation ($n$). 

\begin{figure}[h]
\tikzmath{\th=.5;
\ccx=-6; \ccy=0;} 
\begin{tikzpicture}
\draw (-2.5,-\th/2) node[thick, fill=yellow] {$v_1$};
\draw (-.5,-\th/2) node[thick, fill=yellow] {$v_n$};
\node at (-1.5,-\th/2){$\cdots$};
\node at (1.5,-\th/2){$\cdots$};
\begin{knot}[%draft mode = crossings,
    clip width = 4,
    flip crossing = 2
    ]
\strand[very thick]
    (-3.5,0) to (0,0);
\strand[very thick]
    (-3.5,-\th) to (0,-\th);
\strand[very thick] (0,0)
    to[out=00,in=180] (1,-\th);
\strand[very thick] (2,-\th)
    to[out=00,in=180] (3,0);
\strand[very thick] (0,-\th)
    to[out=00,in=180] (1,-0);
\strand[very thick] (2,-0)
    to[out=00,in=180] (3,-\th);
%loops
\strand[very thick,rounded corners=.2cm] (3,0) to (3.25,0)
        .. controls (3+2.5,2) and (-3.5-2.5,2) .. (-3.75,0)
        to (-3.5,0);
\strand[very thick,rounded corners=.2cm] (3,-\th) to (3.25,-\th)
        .. controls (3+2.5,2-\th) and (-3.5-2.5,2-\th) .. (-3.75,-\th)
        to (-3.5,-\th);
\end{knot}
\draw[dotted] (-3.25,-1/4) circle (1/6);
\draw[dotted] (\ccx, \ccy-1) to[in=-90,out=-90] (-3.25,-1/4-1/6);
\draw[dotted] (\ccx,\ccy) circle (1);
\draw[very thick, ->] (\ccx-1/2,\ccy-1/2-.021) to (\ccx-1/2,\ccy+1-1/2);
\draw[very thick, ->] (\ccx-1/2,\ccy-1/2) to (\ccx+1-1/2,\ccy-1/2);
\node[anchor=south] at (\ccx+1-1/2,\ccy-1/2) {\small{($c$)}};
\node[anchor=west] at (\ccx-1/2,\ccy+1-1/2) {\small{($n$)}};
\end{tikzpicture}
\caption{The ribbon associated to $v$.}
\label{fig:ribbonv}
\end{figure}
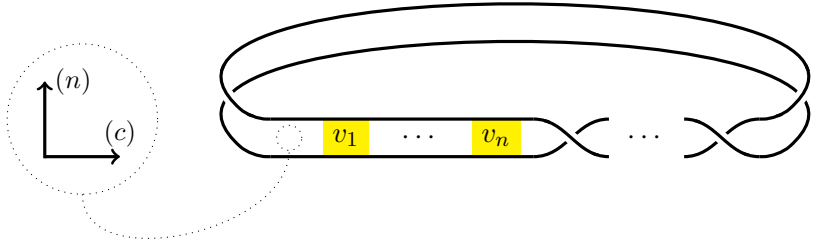

\item For every edge $(v,v')\in\calE$, we plumb the ribbons for $v$ and $v'$ along the squares for $v'$ and $v$ such that the core orientation of the ribbon for $v$ matches the normal orientation of the ribbon for $v'$ and vice versa. 

\item The plumbed ribbons define a surface and the boundary of this surface is a link $L$. We say that $L$ is the {\it arborescent link} associated to $\Gamma$.
\end{enumerate}

\begin{remark} \label{reduced} (i)  If $\Gamma$ is alternating, then so is the arborescent link $L$ by (1)--(3).

(ii) A weighted tree $\Gamma$ is {\it reduced} if it has no vertex of degree at most $2$ and weight $0$ and does not consist of a single vertex of weight $\pm 1$. According to \cite[Section 17.3 (0.2)]{afhkln}, any vertex with weight $0$ and degree 2 can be removed from $\Gamma$ without changing $L$ and by \cite[Section 17.5.3]{afhkln}, this condition implies that $L$ is prime and unsplittable. The graph consisting of a single vertex of weight $\pm 1$ represents the unknot. Throughout this paper, we assume (without further mention) that all weighted trees $\Gamma$ are reduced.
\end{remark}

We now illustrate this construction.

\begin{example}\label{ex:arborescent_knots}
(1) For the weighted tree in Figure \ref{fig:w52}, we construct the ribbons associated to the vertices $v_1$ and $v_2$ in Figure \ref{fig:rib52}. Plumbing the two ribbons in Figure \ref{fig:rib52} yields the surface whose boundary can be transformed into the knot $K=5_2$, see Figure~\ref{fig:52sur}.

(2) We consider the weighted tree in Figure \ref{fig:w85}. The ribbons corresponding to the vertices $v_1$, $v_2$ and $v_3$ are given in Figure \ref{fig:rib85}. Plumbing the three ribbons in Figure \ref{fig:rib85} leads to the surface whose boundary can be transformed into the knot $K=8_5$, see Figure \ref{fig:85sur}.

\begin{figure}[h]
\begin{tikzpicture}
\draw[thick] (1,0) -- (2,0);
\filldraw (1,0) circle (2pt)
 node[anchor=south] {$-2$}
 node[anchor=north] {$v_1$};
\filldraw (2,0) circle (2pt)
 node[anchor=south] {$3$}
 node[anchor=north] {$v_2$};
\end{tikzpicture}
\caption{A weighted tree for $5_2$.}
\label{fig:w52}
\end{figure}
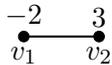

\begin{figure}[h]
\includegraphics[scale=.85]{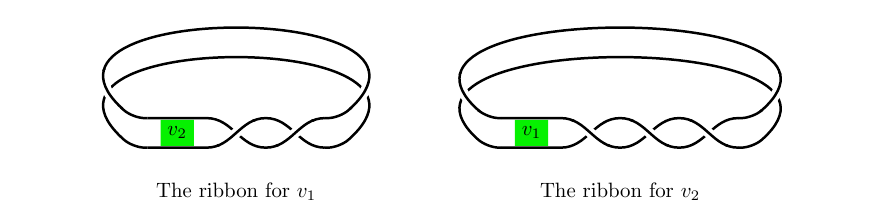}
\caption{The two ribbons for Figure \ref{fig:w52}.}
\label{fig:rib52}
\end{figure}

\begin{figure}[h]
\includegraphics[width = .5\textwidth]{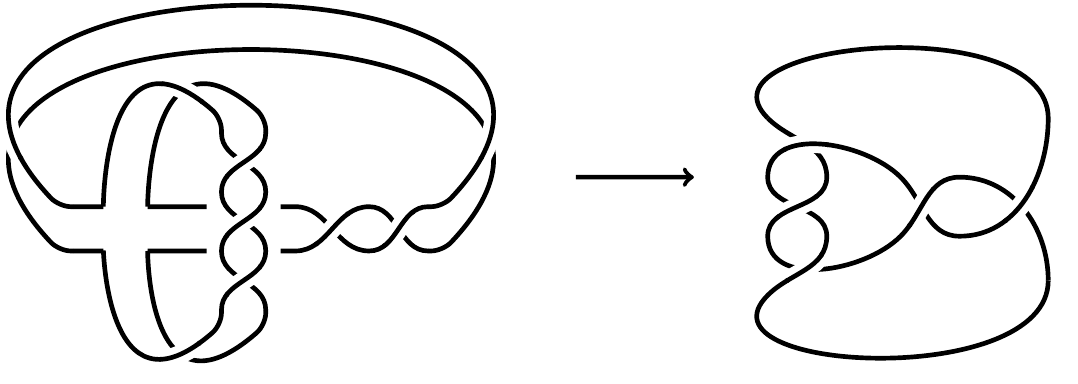} 
\caption{The surface for the two plumbed ribbons and $5_2$.}
\label{fig:52sur}
\end{figure}

\begin{figure}[h]
\begin{tikzpicture}[scale=.80]
\node (v0) at (0,0){};
\node (v1) at (-3/2,0){};
\node (v2) at (0,-3/2){};
\node (v3) at (3/2,0){};

\filldraw [] (v1) circle (3pt)
  node[anchor=north east] {$v_1$}
  node[anchor=south east] {$3$};
\filldraw [black] (v0) circle (3pt)
  node[anchor=south west] {$v_0$}
   node[anchor=south east] {$-0$};
\filldraw [] (v3) circle (3pt)
  node[anchor=south west] {$2$}
  node[anchor=north west] {$v_3$};
\filldraw [] (v2) circle (3pt)
  node[anchor=north east] {$3$}
  node[anchor=north west] {$v_2$};
\draw[very thick] (v1.center) -- (v3.center);
\draw[very thick] (v2.center) -- (v0.center);
\end{tikzpicture}
\caption{A weighted tree for $8_5$.}
\label{fig:w85}
\end{figure}
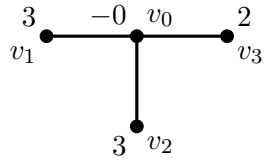

\begin{figure}[H]\begin{center}
\includegraphics[scale=.85]
{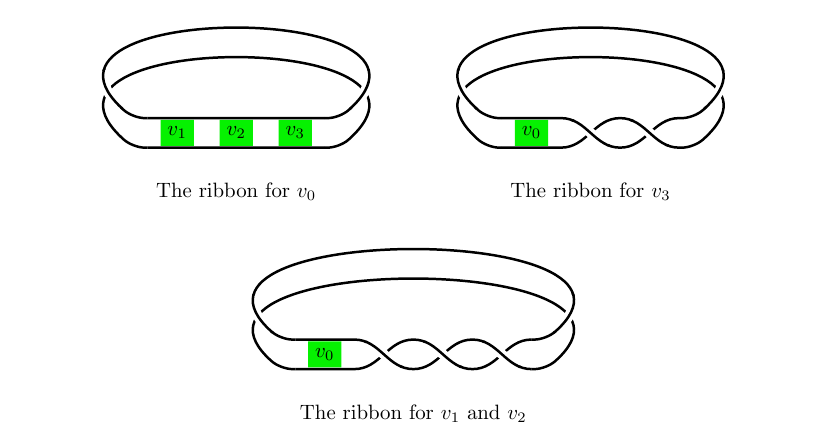}
\end{center}
\caption{The three ribbons for Figure \ref{fig:w85}.}
\label{fig:rib85}
\end{figure}

\begin{figure}[h]
\includegraphics[scale=.9]{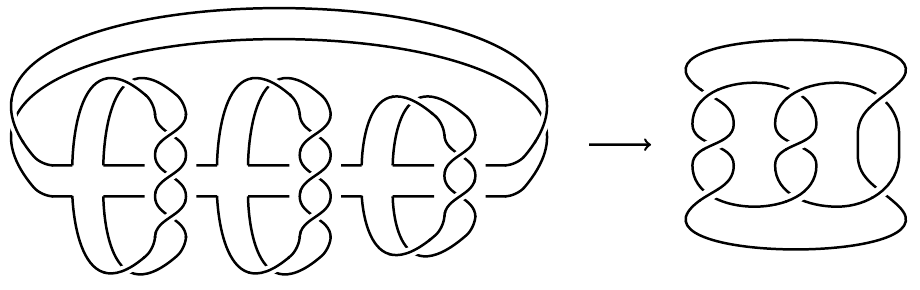}
\caption{The surface for the three plumbed ribbons and $8_5$.}
\label{fig:85sur}
\end{figure}
\end{example}

\subsection{Arborescent tangles}
\label{sec:arbtangl}
A tangle is defined as a region of a link diagram with exactly four emerging strings in the directions NW, NE, SE and SW. Two tangles are equivalent if one can be deformed into the other via a sequence of Reidemeister moves without changing the four emerging strings or moving another string over the emerging strings. A \emph{weighted rooted tree} is a weighted tree with a marked vertex, the \emph{root}, which has an emanating germ in one direction, here depicted by
    \tikz{
    \filldraw[fill=black] (1,0) circle (2pt);
    \draw[thick] (1,0) -- (1,.25);
    }.
For a weighted rooted tree with root $v_0$, the associated weighted tree is the tree where $v_0$ is considered as an ordinary vertex.
In this case, we define an \emph{arborescent tangle} as the boundary of the surface corresponding to the associated weighted tree where the ribbon for $v_0$ is cut at the place corresponding to the germ, leaving four emanating strings. Consider the weighted tree from Example~\ref{ex:arborescent_knots} (1) with root $v_1$ as in Figure \ref{fig:w52r}. The corresponding tangle is given by Figure~\ref{fig:52cut}.
\begin{figure}[ht]
\begin{tikzpicture}
\draw[thick] (1,0) -- (2,0);
\draw[thick] (1,0) -- (1-.25,0);
\filldraw[fill=black] (1,0) circle (2pt)
 node[anchor=south] {$-2$}
 node[anchor=north] {$v_1$};
\filldraw (2,0) circle (2pt)
 node[anchor=south] {$3$}
 node[anchor=north] {$v_2$};
\end{tikzpicture}
\caption{A weighted rooted tree for $5_2$.}
\label{fig:w52r}
\end{figure}
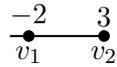
\begin{figure}[ht]
\centering
\includegraphics[width=.6\textwidth]{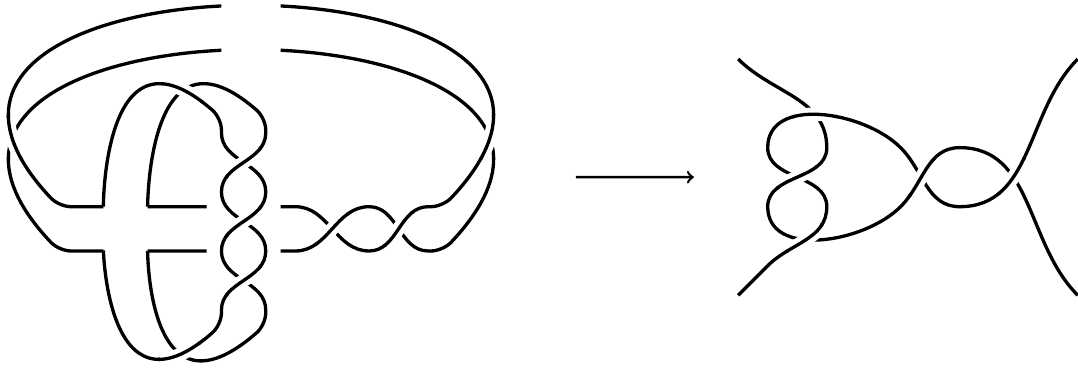}
\caption{The tangle for Figure \ref{fig:w52r}.}
\label{fig:52cut}
\end{figure}

We extend the notion of alternating from weighted trees to weighted rooted trees and adapt the bipartition $\calV = \calV_+ \cup \calV_-$ accordingly. In this case, the corresponding tangle is alternating.

The closure of a tangle is the link $L$ that is obtained by connecting the string in the NW direction with the one in the NE direction as well as the string in the SW direction with the one in the SE direction. For example, the closure of the tangle in Figure~\ref{fig:52cut} is the $5_2$ knot as in Figure~\ref{fig:52sur}. The closure of an arborescent tangle is the arborescent link constructed from the associated weighted tree.

The $\pm$-Tait graphs $\mathcal{T}_{\pm}$ for an alternating tangle are defined as in Section \ref{sec:tait_graph} with the addition of marked vertices corresponding to the North and South faces or the East and West faces. They are depicted by~\tikz{
    \filldraw[fill=white] (0,0) circle (3pt) node {$\star$};
}.
The reduced Tait graphs $\mathcal{T}_{\pm}^{\prime}$ of a tangle are obtained from the Tait graphs $\mathcal{T}_{\pm}$ of a tangle by replacing multiple edges with single edges and removing loops.

\begin{remark} \label{T2T}
The (reduced) Tait graphs for the closure $L$ of a tangle can be obtained from the (reduced) Tait graphs of the tangle as follows:\footnote{For $\varepsilon\in\{+,-\}$, we write $-\varepsilon$ to mean the opposite choice of sign as for $\varepsilon$.} Assume that the North and South faces are colored by $\varepsilon\in\{+,-\}$ and thus the East and West faces are colored by $-\varepsilon$. Then the (reduced) $\varepsilon$-Tait graph of $L$ can be obtained from the
(reduced) vertical Tait graph of the tangle by unmarking the marked vertices.
The (reduced) $-\varepsilon$-Tait graph of $L$
is formed by identifying the two marked vertices in the (reduced) horizontal Tait graph of a tangle.
For example, for an arborescent tangle that is constructed from an alternating weighted rooted tree with root $v_0\in\calV_{-\varepsilon}$, the North and South faces are colored by $\varepsilon$.
\end{remark}

\begin{example}\label{ex:tangletait}
(1) Consider the tangle corresponding to the weighted rooted tree in Figure \ref{fig:wrt}.
\begin{figure}[h]
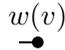

\tikz{
    \filldraw[fill=black] (1,0) circle (2pt);
    \draw[thick] (1,0) -- (1-.25,0);
    \node[anchor=south] at (1,0){$w(v)$};
    }
\caption{A weighted rooted tree for a single vertex.}
\label{fig:wrt}
\end{figure}

\noindent If $w(v)\neq0$ and $\pm = \sign(w(v))$, the Tait graphs $\mathcal{T}_{\pm}$ and $\mathcal{T}_{\mp}$ are given in Figure \ref{fig:Tsign} with $|w(v)|+1$ vertices (including the two marked vertices) and $|w(v)|$ edges, respectively.
 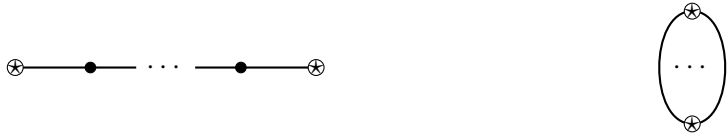
\begin{figure}[h]
 \begin{tikzpicture}
  \node (dots) at (2,0){$\cdots$};
  \draw[thick] (0,0) -- (dots) -- (4,0);
  \filldraw[fill=white] (0,0) circle (3pt) node {$\star$};
  \filldraw[fill=black] (1,0) circle (2pt);
  \filldraw[fill=black] (3,0) circle (2pt);
  \filldraw[fill=white] (4,0) circle (3pt) node {$\star$};
  \draw[thick] (9,3/4) to[out=0,in=0] (9,-3/4);
  \draw[thick] (9,3/4) to[out=180,in=180] (9,-3/4);
  \filldraw[fill=white] (9,3/4) circle (3pt) node {$\star$};
  \node (dots) at (9,0){$\cdots$};
  \filldraw[fill=white] (9,-3/4) circle (3pt) node {$\star$};
 \end{tikzpicture}
 \caption{The Tait graphs $\mathcal{T}_{\pm}$ (left) and $\mathcal{T}_{\mp}$ (right) for Figure \ref{fig:wrt}.}
\label{fig:Tsign}
 \end{figure}

\noindent (2) The Tait graphs for the tangle in Figure \ref{fig:52cut} are given in Figure \ref{fig:T52}. From Figure \ref{fig:T52}, we obtain the Tait graphs for $5_2$ as depicted in Figure \ref{fig:dTrT} by unmarking the marked vertices for $\calT_+$ and identifying the marked vertices for $\calT_-$.
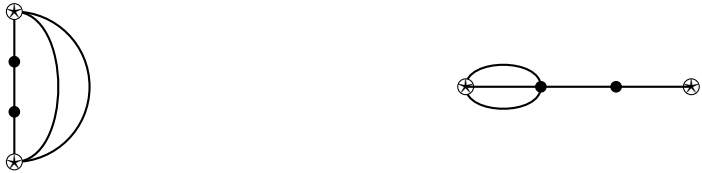
\begin{figure}[H]
 \begin{tikzpicture}
  \draw[thick] (0,1) -- (0,-1);
  \draw[thick] (0,1) to[out=0,in=0] (0,-1);
  \draw[thick] (0,1) to[out=0,in=90] (1,0)
    to[out=-90,in=0] (0,-1);
  \filldraw[fill=white] (0,1) circle (3pt) node {$\star$};
  \filldraw[fill=black] (0,1/3) circle (2pt);
  \filldraw[fill=black] (0,-1/3) circle (2pt);
  \filldraw[fill=white] (0,-1) circle (3pt) node {$\star$};
  \draw[thick] (6,0) to[out=90,in=90] (7,0);
  \draw[thick] (6,0) to[out=-90,in=-90] (7,0);
  \filldraw[fill=white] (6,0) circle (3pt) node {$\star$};
  \filldraw[fill=black] (7,0) circle (2pt);
  \filldraw[fill=black] (8,0) circle (2pt);
  \filldraw[fill=white] (9,0) circle (3pt) node {$\star$};
  \draw[thick] (6,0) -- (9,0);
 \end{tikzpicture}
\caption{The Tait graphs $\mathcal{T}_{+}$ (left) and $\mathcal{T}_{-}$ (right) for Figure \ref{fig:52cut}.}
\label{fig:T52}
\end{figure}
\end{example}

\section{Proof of Theorem~\ref{main}}

In order to prove Theorem~\ref{main}, we need the following four results. We depict tangles $T$ (with indices) by circles, weighted trees $\Gamma$ (with indices) by squares and Tait graphs $\mathcal{T}$ (with indices) by hexagons. We also denote North-South (respectively, East-West) Tait graphs with marked vertices as in Figure \ref{fig:NSEW}.
\begin{figure}[ht]
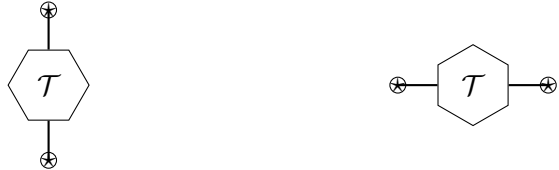

\tikz{
\useasboundingbox (-1,-1) rectangle (1,1);
\node (north) at (0,1){};
\node (south) at (0,-1){};
\filldraw[fill=white] (north) circle (3pt) node {$\star$};
\filldraw[fill=white] (south) circle (3pt) node {$\star$};
\node[rotate=0,regular polygon,regular polygon sides=6, draw,inner sep=5]
(Tn) at (0,0)
{$\calT$};
\draw[thick] (north.center) -- (Tn.north);
\draw[thick] (south.center) -- (Tn.south);
}\hspace{100pt}
\tikz{
\useasboundingbox (-1,-1) rectangle (1,1);
\node (north) at (1,0){};
\node (south) at (-1,0){};
\filldraw[fill=white] (north) circle (3pt) node {$\star$};
\filldraw[fill=white] (south) circle (3pt) node {$\star$};
\node[rotate=-90,regular polygon,regular polygon sides=6, draw,inner sep=5]
(Tn) at (0,0)
{\rotatebox{90}{$\calT$}};
\draw[thick] (north.center) -- (Tn.north);
\draw[thick] (south.center) -- (Tn.south);
}
\caption{North-South (left) and East-West Tait (right) graphs.}
\label{fig:NSEW}
\end{figure}

\begin{proposition}\label{prp:taitgarph_arbo}
Let $\Gamma$ be an alternating weighted rooted tree with root $v_0$ and associated tangle $T$. Let $v_0\in\mathcal V_{\varepsilon}$ be connected to the subgraphs $\Gamma_1,\ldots,\Gamma_n$ via the vertices $v_1,\ldots,v_n\in\calV_{-\varepsilon}$ {in counterclockwise order} as in Figure~\ref{fig:rib}.
Let $\calT^{\, i}_{\pm}$ denote the Tait graphs of the tangle corresponding to the tree $\Gamma_i$ with root $v_i$ for $i=1, \dotsc, n$. Then the Tait graph $\calT_{\varepsilon}$ of $\Gamma$ is given by Figure \ref{fig:NSTait} (top) with $|w(v_0)|+1$ additional vertices on the right, including the marked vertex. Moreover, the Tait graph $\mathcal T_{-\varepsilon}$ of $\Gamma$ is given by Figure \ref{fig:NSTait} (bottom) with $|w(v_0)|$ additional edges from the top vertex to the bottom vertex.

\begin{figure}[h]
\begin{tikzpicture}
  \node (w1) at (-0.75,0){};
  \filldraw (w1) circle (2pt) node[anchor=south east]{$v_0$};
  \draw[thick] (-0.75,0) -- (-0.75,.25);

  \node[regular polygon,regular polygon sides=4, draw, outer sep=0,minimum size=1cm] (v1) at (-3,-2){}; \node at (v1) {$\Gamma_{1}$};
  \filldraw (v1.north) circle (2pt) node[anchor=south east]{$v_{1}$};
  
  \node[regular polygon,regular polygon sides=4, draw, outer sep=0, minimum size=1cm] (v2) at (-1.5,-2){}; \node at (v2) {$\Gamma_{2}$};
  \filldraw (v2.north) circle (2pt) node[anchor=south east]{$v_{2}$};

  \node[regular polygon,regular polygon sides=4, draw, outer sep=0, minimum size=1cm] (vm) at (1.5,-2){};
  \node at (vm) {$\Gamma_{n}$};
  \filldraw (vm.north) circle (2pt) node[anchor=south west]{$v_{n}$};

  \draw[thick] (w1.center) -- (v1.north);
  \draw[thick] (w1.center) -- (v2.north);
  \draw[thick] (w1.center) -- (vm.north);
  \node[anchor=center] at (0,-2) {$\ldots$};
\end{tikzpicture}
\caption{$\Gamma$ with root $v_0$.}
\label{fig:rib}
\end{figure}

\begin{figure}[h]
\begin{tikzpicture}
\node (west) at (.5,0){};
\node (east) at (8.75,0){};

\node[rotate=-90,regular polygon,regular polygon sides=6, draw,inner sep= 3] (T1) at (1.5,0) {$\calT^{\,1}_{\varepsilon}$};
\node[rotate=-90,regular polygon,regular polygon sides=6, draw,inner sep= 3] (T2) at (3,0) {$\calT^{\, 2}_{\varepsilon}$};
\node (dots) at (4,0) {$\ldots$};
\node[rotate=-90,regular polygon,regular polygon sides=6, draw,inner sep= 3] (Tn) at (5,0) {$\calT^{\, n}_{\varepsilon}$};
  
\filldraw[fill] (5.75,0) circle (2pt);
\filldraw[fill] (6.25,0) circle (2pt);
\filldraw[fill] (7.75,0) circle (2pt);
\filldraw[fill] (8.25,0) circle (2pt);
\node (dots2) at (7,0) {$\ldots$};

\draw[thick] (west.center) -- (T1)-- (T2) -- (dots) -- (Tn) -- (dots2) -- (east.center);
  
\filldraw[fill=white] (east) circle (3pt) node {$\star$};
\filldraw[fill=white] (west) circle (3pt) node {$\star$};
\filldraw (2.25,0) circle (2pt);
 
\draw [
    thick, decorate,
    decoration={brace,mirror,raise=10pt}]
    (Tn)+(5/8,0) -- (8.3,0)
    node at (7.05,-.75){$|w(v_0)|$}; 
\end{tikzpicture}
\\ \vspace{12pt}
\begin{tikzpicture}
\node (north) at (.5,2){};
\node (south) at (.5,-2){};
\node[regular polygon,regular polygon sides=6, draw,inner sep= 3] (T2) at (-.5,0) {\rotatebox{-90}{$\calT^{\, 1}_{{-\varepsilon}}$}};
\node at (.5,0) {$\ldots$};
\node[regular polygon,regular polygon sides=6, draw,inner sep= 3] (Tn) at (1.5,0) {\rotatebox{-90}{$\calT^{\, n}_{{-\varepsilon}}$}};
\draw[thick] (north.center) -- (T2.north);
\draw[thick] (north.center) -- (Tn.north);
\draw[thick] (south.center) -- (T2.south);
\draw[thick] (south.center) -- (Tn.south);
\draw[thick] (north.center) to[out = 0, in=90] (2.5,0) to[out=-90,in=0] (south.center)[in=-20];
\node at (3,0) {$\ldots$};
\draw[thick] (north.center) to[out = 0, in=90] (3.5,0) to[out=-90,in=0] (south.center)[in=-20];
\filldraw[fill=white] (north) circle (3pt) node[scale=1] {$\star$};
\filldraw[fill=white] (south) circle (3pt) node {$\star$};
\end{tikzpicture}
\caption{The Tait graph $\calT_{\varepsilon}$ (top) and the Tait graph $\calT_{-\varepsilon}$ (bottom).}
\label{fig:NSTait}
\end{figure}
\end{proposition}

\begin{proof}
Consider the ribbon corresponding to $v_0$ with gluing points for $v_1,\ldots,v_n$ followed by $w(v_0)$ half-twists. We assume that $v_0\in\calV_+$, i.e., $\varepsilon = +$ and $w(v_0)\geq0$. A similar argument holds for $v_0\in\calV_{-}$. As $\Gamma$ is alternating, $v_1,\ldots,v_n\in\calV_-$ and so $w(v_i)\leq 0$. Let $T_i$ denote the tangle for $\Gamma_i$. Then the tangle associated to $\Gamma_i$ with root $v_i$ is in Figure \ref{fig:gamma}. The tangle for $\Gamma$ with root $v_0$ with marked squares for $v_1,\ldots,v_n$ is in Figure~\ref{fig:tangle}. We then glue the tangles associated with $v_1,\ldots,v_n$ to the one for $v_0$ as in Section~\ref{ak} and rearrange the strands to obtain the tangle for $\Gamma$, see Figure \ref{fig:tangleGamma}. Using Figure \ref{fig:cross} and Remark \ref{T2T}, the Tait graphs $\mathcal{T}_{\pm}$ for Figure \ref{fig:tangleGamma} are of the shape in Figure~\ref{fig:NSTait}.
\end{proof}

\begin{figure}[ht]
\includegraphics[scale=1]{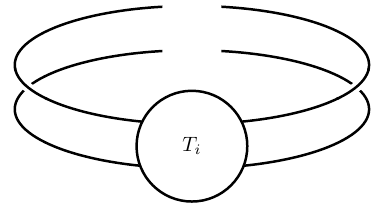}
\caption{Tangle for $\Gamma_i$ with root $v_i$.}
\label{fig:gamma}
\end{figure}

\begin{figure}[ht]
\includegraphics[scale=1]{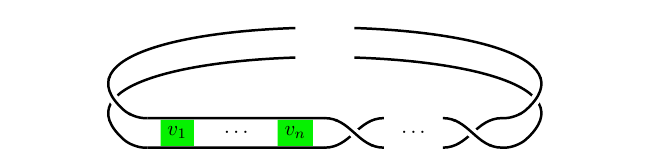}
\caption{Tangle for $\Gamma$ with root $v_0$.}
\label{fig:tangle}
\end{figure}

\begin{figure}[H]
\begin{tikzpicture}
\node[circle,draw, very thick,anchor=center]
  (T1) at (-4.5,0) {\rotatebox{270}{\tiny $T_1$}};
\node[circle,draw, very thick,anchor=center]
  (T2) at (-3,-0) {\rotatebox{270}{\tiny $T_2$}};
\node[%draw,
    very thick,anchor=center,
    minimum width=0.5cm,
    minimum height = 8mm
    ]
  (cdots) at (-2,-0) {$\cdots$};
\node[circle,draw, very thick,anchor=center]
  (Tn) at (-1,0) {
     \rotatebox{270}{
     \tiny $T_n$
     }
  };

\draw[very thick] (T1.north east) to[out=45,in=+90+45] (T2.north west);
\draw[very thick] (T1.south east) to[out=-45,in=-90-45] (T2.south west);
\draw[very thick] (T2.north east) to[out=45,in=-180] (cdots.north west);
\draw[very thick] (T2.south east) to[out=-45,in=-180] (cdots.south west);
\draw[very thick] (Tn.north west) to[out=-90-45-90,in=0] (cdots.north east);
\draw[very thick] (Tn.south west) to[out=90+90+45,in=0] (cdots.south east);

\draw[very thick] (T1.north west) to[out=-45-180,in=-45] (-5.5,1);
\draw[very thick] (T1.south west) to[out=90+90+45,in=45] (-5.5,-1);

\begin{knot}[%draft mode = crossings,
    clip width = 4,
    flip crossing/.list={2,4,3}
    ]
% v1
\strand[very thick]
    (Tn.north east)
    to[out=45,in=180] (1,-.25)
    to[out=00,in=180] (2,.25);
\strand[very thick]
    (Tn.south east)
    to[out=-45,in=180] (1,.25)
    to[out=-00,in=180] (2,-.25);
\strand[very thick]
    (3,-.25)
    to[out=0,in=180+45] (4.5,1);
\strand[very thick]
    (3,.25)
    to[out=0,in=180-45] (4.5,-1);
\end{knot}
 \node at (2.5,0) {$\ldots$};
\end{tikzpicture}
\caption{The tangle for $\Gamma$ with root $v_0$.}
\label{fig:tangleGamma}
\end{figure}
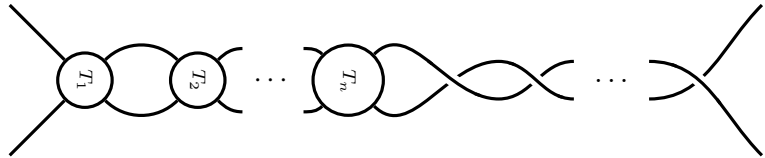

\begin{example}
According to Example~\ref{ex:tangletait} (1), the Tait graphs for
 \tikz{
    \filldraw[fill=black] (1,0) circle (2pt);
    \draw[thick] (1,0) -- (1-.25,0);
    \node[anchor=south] at (1,0){$3$};
    }
 are given in Figure \ref{fig:T23}. Proposition~\ref{prp:taitgarph_arbo} implies that the Tait graphs for the rooted tree in Figure \ref{fig:w52r} are given in Figure~\ref{fig:Tw52r} in accordance with Figure \ref{fig:T52}.
 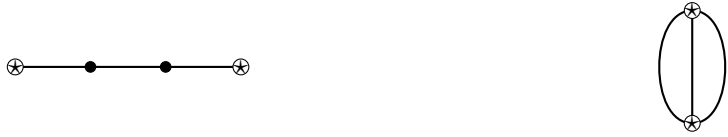
\begin{figure}[H]
\begin{tikzpicture}
  \draw[thick] (0,0) -- (3,0);
  \filldraw[fill=white] (0,0) circle (3pt) node {$\star$};
  \filldraw[fill=black] (1,0) circle (2pt);
  \filldraw[fill=black] (2,0) circle (2pt);
  \filldraw[fill=white] (3,0) circle (3pt) node {$\star$};
  \draw[thick] (9,3/4) to[out=0,in=0] (9,-3/4);
  \draw[thick] (9,3/4) to[out=180,in=180] (9,-3/4);
  \draw[thick] (9,3/4) to (9,-3/4);
  \filldraw[fill=white] (9,3/4) circle (3pt) node {$\star$};
  \filldraw[fill=white] (9,-3/4) circle (3pt) node {$\star$};
 \end{tikzpicture}
 \caption{The Tait graphs $\mathcal T_{+}^{\,1}$ (left) and $\mathcal T_{-}^{\,1}$ (right) for the root $w(v_1) = 3$.}
\label{fig:T23}
 \end{figure}
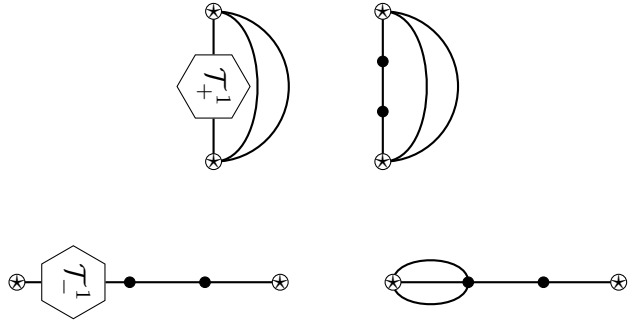
\begin{figure}[H]
  \begin{tikzpicture}
  \node[regular polygon,regular polygon sides=6, draw,inner sep= 1] (T1) at (0,0) {\rotatebox{-90}{$\calT^{\,1}_{+}$}};
  \draw[thick] (0,1) to (T1) to (0,-1);
  \draw[thick] (0,1) to[out=0,in=0] (0,-1);
  \draw[thick] (0,1) to[out=0,in=90] (1,0) to[out=-90,in=0] (0,-1);
  \filldraw[fill=white] (0,1) circle (3pt) node {$\star$};
  \filldraw[fill=white] (0,-1) circle (3pt) node {$\star$};
  \draw[thick] (2.25,1) -- (2.25,-1);
  \draw[thick] (2.25,1) to[out=0,in=0] (2.25,-1);
  \draw[thick] (2.25,1) to[out=0,in=90] (2.25+1,0)
  to[out=-90,in=0] (2.25+0,-1);
  \filldraw[fill=white] (2.25,1) circle (3pt) node {$\star$};
  \filldraw[fill=black] (2.25,1/3) circle (2pt);
  \filldraw[fill=black] (2.25,-1/3) circle (2pt);
  \filldraw[fill=white] (2.25,-1) circle (3pt) node {$\star$};
 \end{tikzpicture}
\\[10pt]
\begin{tikzpicture}
\useasboundingbox (0,-1) rectangle (8,1);
  \node[rotate=-90,regular polygon,regular polygon sides=6, draw,inner sep= 1] (T1) at (.75,0) {$\calT^{\, 1}_{-}$};
  \draw[thick] (0,0) to (T1) to (3.5,0);
  \filldraw[fill=white] (0,0) circle (3pt) node {$\star$};
  \filldraw[fill=black] (1.5,0) circle (2pt);
  \filldraw[fill=black] (2.5,0) circle (2pt);
  \filldraw[fill=white] (3.5,0) circle (3pt) node {$\star$};
\draw[thick] (5,0) to[out=90,in=90] (6,0);
  \draw[thick] (5,0) to[out=-90,in=-90] (6,0);
  \draw[thick] (5,0) -- (8,0);
  \filldraw[fill=white] (5,0) circle (3pt) node {$\star$};
  \filldraw[fill=black] (6,0) circle (2pt);
  \filldraw[fill=black] (7,0) circle (2pt);
  \filldraw[fill=white] (8,0) circle (3pt) node {$\star$};
\end{tikzpicture}
\caption{The Tait graphs $\mathcal T_{+}$ (top) and $\mathcal T_{-}$ (bottom) for Figure \ref{fig:w52r}.}
\label{fig:Tw52r}
\end{figure}
 \end{example}
 
\begin{lemma}\label{lma:pf1}
Let $\Gamma = (\calV,\calE,w)$ be an alternating weighted rooted tree with $|\calV|\geq 2$ and root $v_0\in\calV_-$ as in Figure \ref{fig:groot}
\begin{figure}[ht]
\begin{tikzpicture}
	\node (v0) at (0,1/2){};
	\draw[thick] (v0.center) -- (0,1/2+1/4);
	\node (w1) at (-3.5,-3/4){};
	\node (wn) at (3.5,-3/4){};
	
	\filldraw (v0) circle (2pt) node[anchor=south east]{$v_0$};
	\filldraw (w1) circle (2pt) node[anchor=south]{$w_1$};
	\filldraw (wn) circle (2pt) node[anchor=south]{$w_n$};
	
	\node[regular polygon,regular polygon sides=4, draw, outer sep=0] (x11) at (-5,-3) {$\Gamma^{\,{1,1}}$};
	\filldraw (x11.north) circle (2pt) node[anchor=south east]{$x_{1,1}$};
	
	\node[regular polygon,regular polygon sides=4, draw, inner sep=0] (x1m) at (-2,-3) {$\Gamma^{\,{1,m_1}}$};
	\filldraw (x1m.north) circle (2pt) node[anchor=south west]{$x_{1,m_1}$};

	\node[regular polygon,regular polygon sides=4,draw, inner sep=0] (xn1) at (5,-3) {$\Gamma^{\, {n,m_n}}$};
	\filldraw (xn1.north) circle (2pt) node[anchor=south west]{$x_{n,m_n}$};
	
	\node[regular polygon,regular polygon sides=4, draw, outer sep=0] (xnm) at (2,-3) {$\Gamma^{\, {n,1}}$};
	\filldraw (xnm.north) circle (2pt) node[anchor=south east]{$x_{n,1}$};

	\draw[thick] (v0.center) -- (w1.center);
	\draw[thick] (w1.center) -- (x11.north);
	\draw[thick] (w1.center) -- (x1m.north);
	
	\draw[thick] (v0.center) -- (wn.center);
	\draw[thick] (wn.center) -- (xn1.north);
	\draw[thick] (wn.center) -- (xnm.north);
	
	\node[anchor=center] at (0,-1) {$\ldots$};
	\node[anchor=center] at (-3.5,-3) {$\ldots$};
	\node[anchor=center] at (3.5,-3) {$\ldots$};
\end{tikzpicture}
\caption{$\Gamma$ with root $v_0 \in \calV_{-}$.}
\label{fig:groot}
\end{figure}
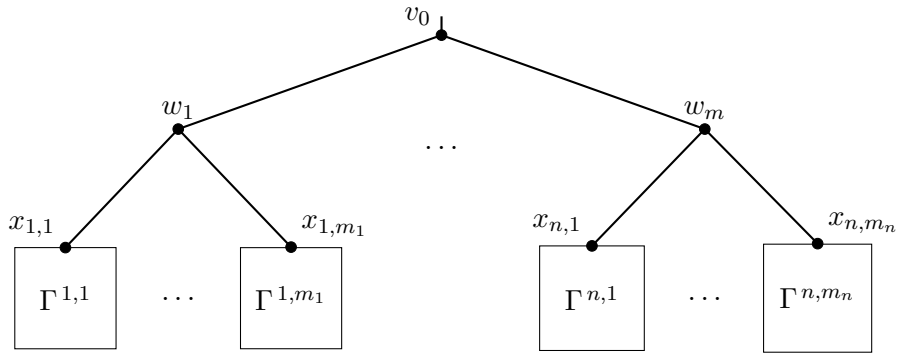
\noindent with weighted subgraphs $\Gamma^{\, {i,j}}$ with roots $x_{i,j}$ and Tait graphs for the corresponding tangles denoted by $\calT_+^{\, {i,j}}$
for $i=1,\ldots,n$ and $j=1,\ldots,m_i$. Then the Tait graph for the tangle associated to $\Gamma$ is given by Figure~\ref{fig:rT} with $|w(v_0)|$ edges from the top vertex to the bottom vertex and $w(w_i)$ additional vertices on the bottom of the $i$th string.
\end{lemma}

\begin{figure}[ht]
\tikzmath{\x1=-2; \x2=2;}
\begin{tikzpicture}[scale=1]
%%% poles
  \node (north) at (0,5.5){};
  \node (south) at (0,-5){};

%%%%%%%%%%%%%%%%%%
%T1 chain:
  \node[regular polygon,regular polygon sides=6, draw=black, inner sep=12, fill=white] (T11) at (-2,3){};
  \node at (T11) {\rotatebox{180}{$\calT^{\,{1,1}}_{+}$}};
  \filldraw (\x1,2) circle (2pt);
  \filldraw (\x1,1) circle (2pt);
  \node (T1dots) at (\x1,1.5){$\vdots$};
  \filldraw (\x1,-1) circle (2pt);
  \node[regular polygon,regular polygon sides=6, draw=black, inner sep=12, fill=white] (T1m) at (\x1,0) {};
  \node at (T1m) {\rotatebox{180}{$\calT^{\, {1,m_1}}_{+}$}};

  \node (w1start) at (\x1,-1){};
  \filldraw (w1start) circle (2pt);
  \filldraw (\x1,-1.5) circle (2pt);
  \node[anchor=center] (w1dots) at (\x1,-2.25) {$\vdots$};
  \filldraw (\x1,-3) circle (2pt);
  \node (w1end) at (\x1,-3.5){};
  \filldraw (w1end) circle (2pt);

  \draw[thick] (north.center) -- (T11.north);
  \draw[thick] (T11.south) -- (T1dots) -- (T1m.north);
  \draw[thick] (T1m.south) -- (w1dots) -- (w1end.center) -- (south.center);

%%%%%%%%%%%%%%%%%%
%Tn chain:
  \node[regular polygon,regular polygon sides=6, draw=black, inner sep=12, fill=white] (Tn1) at (\x2,3){};
  \node at (Tn1) {\rotatebox{180}{$\calT^{\, {n,1}}_{+}$}};
  \filldraw (\x2,2) circle (2pt);
  \filldraw (\x2,1) circle (2pt);
  \node (Tndots) at (\x2,1.5){$\vdots$};
  \filldraw (\x2,-1) circle (2pt);
  \node[regular polygon,regular polygon sides=6, draw=black, inner sep=12, fill=white] (Tnm) at (\x2,0){};
  \node at (Tnm) {\rotatebox{180}{$\calT^{\, {n,m_n}}_{+}$}};

  \node (wnstart) at (\x2,-1){};
  \filldraw (wnstart) circle (2pt);
  \filldraw (\x2,-1.5) circle (2pt);
  \node[anchor=center] (wndots) at (\x2,-2.25) {$\vdots$};
  \filldraw (\x2,-3) circle (2pt);
  \node (wnend) at (\x2,-3.5){};
  \filldraw (wnend) circle (2pt);

  \draw[thick] (north.center) -- (Tn1.north);
  \draw[thick] (Tn1.south) -- (Tndots) -- (Tnm.north);
  \draw[thick] (Tnm.south) -- (wndots) -- (wnend.center) -- (south.center);

%%%%%%%%%%%%%%%%
%%% ldots
  \node[anchor = center] at (0,3) {$\ldots$};
  \node[anchor = center] at (0,0) {$\ldots$};

%%%%%%%%%%%%%%%%%
% additional edge
\draw[thick] (north.center) to[out = 0, in=90] (5,0) to[out=-90,in=0] (south.center)[in=-20];
\draw[thick] (north.center) to[out = 0, in=90] (6.5,0) to[out=-90,in=0] (south.center)[in=-20];
\node[anchor = center] at (5.75,0) {$\cdots$};

  \filldraw[fill=white] (north) circle (3pt) node[scale=1] {$\star$};
  \filldraw[fill=white] (south) circle (3pt) node {$\star$};

  %%%braces
\draw [thick,decorate,decoration={brace}]
   (\x1-.25,-3.6) -- (\x1-.25,-.8)  node at (\x1-1,-2.25) {$w(w_1)$};
\draw [thick,decorate,decoration={brace}]
   (\x2-.25,-3.6) -- (\x2-.25,-.8)  node at (\x2-1,-2.25) {$w(w_n)$};
\end{tikzpicture}
\caption{Tait graph for $\Gamma$.}
\label{fig:rT}
\end{figure}
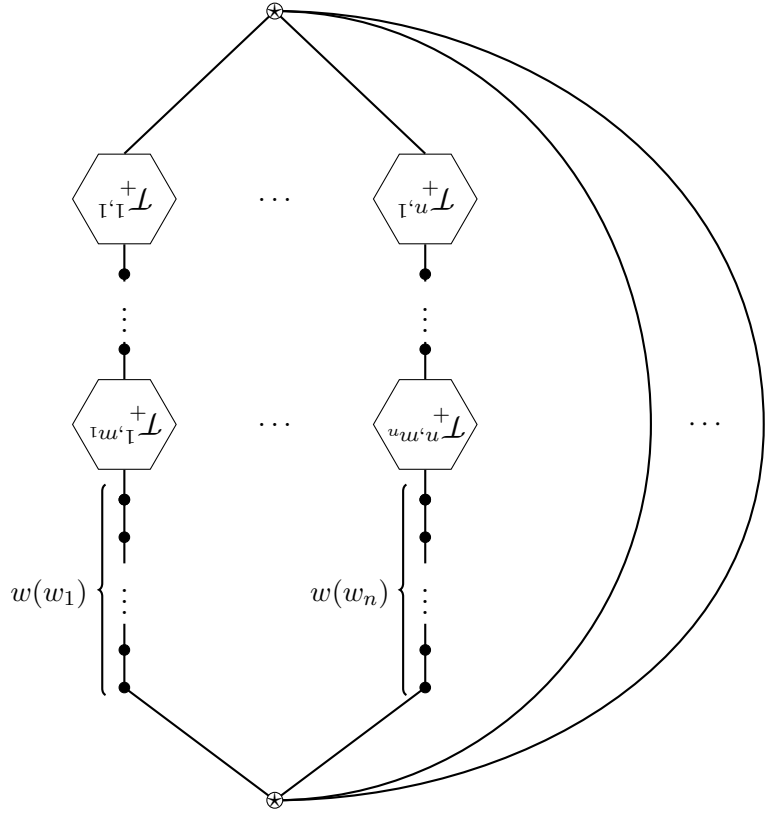

\begin{remark}\label{leafcase}
 If $w_i$ is a leaf, i.e., $m_i=0$ for some $i=1,\ldots,n$ in Figure \ref{fig:groot}, then the $w(w_i)$ additional vertices on the bottom include the top marked vertex in Figure \ref{fig:rT}.
\end{remark}

\begin{proof}
Consider the weighted subtree with root $w_i$ in Figure \ref{fig:wsub}. By Proposition~\ref{prp:taitgarph_arbo}, the Tait graph $\calT^{\, i}_+$ corresponding to Figure \ref{fig:wsub} is given by Figure \ref{fig:Twsub} where $\calT_+^{\, {i,j}}$ are the Tait graphs for $\Gamma^{\, {i,j}}$ with roots $x_{i,j}$. We now apply Proposition~\ref{prp:taitgarph_arbo} to Figure \ref{fig:groot} and obtain that its Tait graph is given by Figure \ref{fig:Tgroot} with $|w(v_0)|$ edges between the top vertex and bottom vertex. Inserting Figure~\ref{fig:Twsub} into Figure~\ref{fig:Tgroot} yields Figure \ref{fig:rT}.
\end{proof}
%\vspace{12.2pt}

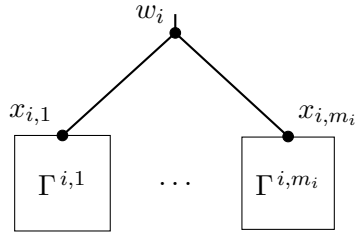
\begin{figure}[h]
\begin{tikzpicture}
  \node (w1) at (-3.5,-1){};
  \draw[thick] (w1.center) -- (-3.5,-1+1/4);
  \filldraw (w1) circle (2pt) node[anchor=south east]{$w_i$};

  \node[regular polygon,regular polygon sides=4, draw, outer sep=0] (x11) at (-5,-3) {$\Gamma^{\, {i,1}}$};
  \filldraw (x11.north) circle (2pt) node[anchor=south east]{$x_{i,1}$};

  \node[regular polygon,regular polygon sides=4, draw, inner sep=0] (x1m) at (-2,-3) {$\Gamma^{\, {i,m_i}}$};
  \filldraw (x1m.north) circle (2pt) node[anchor=south west]{$x_{i,m_i}$};

  \draw[thick] (w1.center) -- (x11.north);
  \draw[thick] (w1.center) -- (x1m.north);
  \node[anchor=center] at (-3.5,-3) {$\ldots$};
\end{tikzpicture}
\caption{Weighted subtree with root $w_i$.}
\label{fig:wsub}
\end{figure}

\begin{figure}[t]
\begin{tikzpicture}
\node (west) at (.5,0){};
\node (east) at (8.25,0){};
\node[rotate=-90,regular polygon,regular polygon sides=6, draw,inner sep= 10] 
   (T1) at (1.5,0){};
\node[rotate=-90] at (T1) {$\calT^{\, {i,1}}_{+}$};
\node[rotate=-90,regular polygon,regular polygon sides=6, draw,inner sep= 10]
   (T2) at (3,0){};
\node[rotate=-90] at (T2) {$\calT^{\, {i,2}}_{+}$};
\node (dots) at (4.25,0) {$\ldots$};
\node[rotate=-90,regular polygon,regular polygon sides=6, draw,inner sep=10]
   (Tn) at (5.5,0){};
\node[rotate=-90] at (Tn){$\calT^{\, \tiny{i,m_i}}_{+}$};

\filldraw (2.25,0) circle (2pt);
\filldraw (3.75,0) circle (2pt);
\filldraw (4.75,0) circle (2pt);

\filldraw[fill] (6.25,0) circle (2pt);
\filldraw[fill] (6.75,0) circle (2pt);
\filldraw[fill] (7.75,0) circle (2pt);
\node (dots2) at (7.25,0) {$\ldots$};

\draw[thick] (west.center) -- (T1)-- (T2) -- (dots) -- (Tn) -- (dots2) -- (east.center);
  
\filldraw[fill=white] (east) circle (3pt) node {$\star$};
\filldraw[fill=white] (west) circle (3pt) node {$\star$};

\draw [
    thick, decorate,
    decoration={brace,mirror,raise=10pt}]
    (Tn)+(5/8,0) -- (7.9,0)
    node at (7.10,-.75){$w(w_i)$}; 
\end{tikzpicture}
\caption{Tait graph $\mathcal{T}^{\, i}_{+}$ for Figure \ref{fig:wsub}.}
\label{fig:Twsub}
\end{figure}
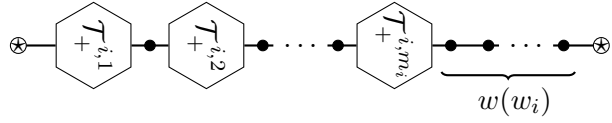

\begin{figure}[h]
  \begin{tikzpicture}  
%%%%%%%%%%%%%%%%
% poles
  \node (north) at (0,5){};
  \node (south) at (0,1){};
%T1:
  \node[regular polygon,regular polygon sides=6, draw,inner sep= 3] (T11) at (-1.5,3) {\rotatebox{-90}{$\calT^{\, 1}_{+}$}};
  \draw[thick] (T11.south) -- (south.center);
  \draw[thick] (T11.north) -- (north.center);
%Tn:
  \node[regular polygon,regular polygon sides=6, draw,inner sep= 3] (Tn1) at (1.5,3) {\rotatebox{-90}{$\calT^{\, n}_{+}$}};
\draw[thick] (Tn1.south) -- (south.center);
 \draw[thick] (Tn1.north) -- (north.center);

\node[anchor = center] at (0,3) {$\ldots$};

%%%%%%%%%%%%%%%%%
% additional edge
  \draw[thick] (north.center) to[out = 0, in=90] (3,3) to[out=-90,in=0] (south.center)[in=-20];
  \draw[thick] (north.center) to[out = 0, in=90] (4,3) to[out=-90,in=0] (south.center)[in=-20];
  \node[anchor = center] at (3.5,3) {$\cdots$};

  \filldraw[fill=white] (north) circle (3pt) node[scale=1] {$\star$};
  \filldraw[fill=white] (south) circle (3pt) node {$\star$};
  \end{tikzpicture}
  \caption{Tait graph of Figure \ref{fig:groot}.}
\label{fig:Tgroot}
\end{figure}
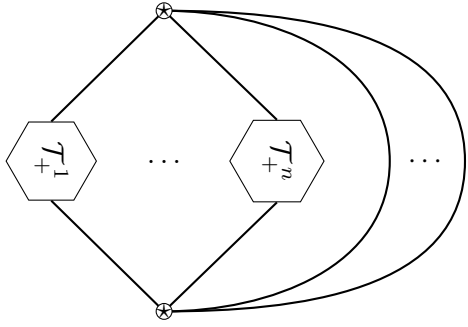

We say that a reduced Tait graph (of a link or tangle) consists of {\it edge-connected polygons} if it can be built by gluing the polygons along a single edge. For example, the reduced Tait graph for $7_4^{*}$ in Figure \ref{fig:rT76} consists of edge-connected polygons while the reduced Tait graph for $8_5$ in Figure \ref{fig:T85} does not as the polygons share two edges instead of one. Moreover, we say that a polygon has size $n\in\Z_{\geq 2}$ if it has $n$ vertices. We follow the convention that a polygon of size $2$ is a graph consisting of two vertices connected by an edge.

\begin{lemma}\label{lma:edgecon}
Let $\Gamma = (\calV,\calE,w)$ be an alternating weighted rooted tree with $|\calV|\geq 2$, root $v_0\in \calV_-$ and $0\notin w(\calV_-)$.
Then the reduced Tait graph $\calT_+'$ of the tangle for $\Gamma$ consists of edge-connected polygons $(P_v)_{v\in\calV_+}$ of sizes $w(v)+e(v)$ for $v\in\calV_+$ and has an edge between the top vertex and bottom vertex.
\end{lemma}

\begin{proof}
We proceed by induction on $|\calV|$. Assume that $|\calV|  = 2$ with \hbox{$\calV_\pm = \{v_\pm\}$} and set $v_-=v_0$. As $\Gamma$ has no leaf of weight $0$ (see Remark \ref{reduced}), we have $w(v_+)\neq 0$. By Example~\ref{ex:tangletait} (1) and Proposition~\ref{prp:taitgarph_arbo}, $\calT_+'$ is a polygon of size $w(v_+)+1$. 

Now assume that the claim is true for all $\Gamma$'s with less than $|\calV|$ vertices. Let $v_0\in\calV_-$ be connected to $w_1,\ldots,w_n\in\calV_+$ and $w_i$ be connected to the subgraphs $\Gamma^{\,{i,j}}$ (with vertices $\calV_{i,j}$) via the vertices $x_{i,j}\in\calV_-$ for $i=1,\ldots,n$ and $j=1,\ldots,m_i$ as in Figure~\ref{fig:groot}. Let $(\calT_+^{\,{i,j}})'$ be the reduced Tait graphs for $\Gamma^{\,{i,j}}$ with roots $x_{i,j}$. By Lemma \ref{lma:pf1}, the reduced Tait graph $\calT_+'$ is given as in Figure~\ref{fig:rT} with exactly one edge from the top vertex to the bottom vertex because $w(v_0)\neq 0$. By assumption, $(\calT_+^{\, {i,j}})'$ consists of edge-connected polygons of sizes $w(v)+e(v)$ for $v\in(\calV_{i,j})_+$ and has an edge between the top and bottom vertices. We denote the reduced Tait graph $(\calT_+^{\,{i,j}})'$ without this additional edge by $\widetilde{\calT}^{\tiny{\,{i,j}}}_{+}$ as depicted in Figure \ref{fig:mT}. 

If we replace the Tait graphs $(\calT^{\,{n,j}}_+)'$ by $\widetilde{\calT}^{{\, n,j}}_+$ for $j=1,\ldots,m_n$ in Figure~\ref{fig:rT} and move the exterior edge across the rightmost string (consisting of the Tait graphs $\widetilde \calT_{+}^{\, {n,j}}$ for $j=1,\ldots,m_n$ and the vertices below them), then $\calT_+'$ is given as in Figure~\ref{fig:final}. The edge from the top vertex 
\begin{figure}[h]
\tikz{
		%%%%%%%%%%%%%%%%%%%%%%%%%%%%
		\node (north) at (0,1){};
		\node (south) at (0,-1){};
		\filldraw[fill=white] (north) circle (3pt) node {$\star$};
		\filldraw[fill=white] (south) circle (3pt) node {$\star$};
		\node[rotate=0,regular polygon,regular polygon sides=6, draw,inner sep=0]
		(Tn) at (0,0)
		{
		\scalebox{.8}{
		$(\calT^{\tiny{\,i,j}}_{+})'$}};
		\draw[thick] (north.center) -- (Tn.north);
		\draw[thick] (south.center) -- (Tn.south);
		%%%%%%%%%%%%%%%%%%
		\node at (1.1,0) {$=$};
		\node (north2) at (2,1){};
		\node (south2) at (2,-1){};
		\filldraw[fill=white] (north2) circle (3pt) node {$\star$};
		\filldraw[fill=white] (south2) circle (3pt) node {$\star$};
		\node[rotate=0,regular polygon,regular polygon sides=6, draw,inner sep=0]
		(Tn2) at (2,0)
		{$\widetilde{\calT}^{\tiny{\,i,j}}_{+}$};
		\draw[thick] (north2.center) -- (Tn2.north);
		\draw[thick] (south2.center) -- (Tn2.south);
		\draw[thick] (north2.center)[out=0,in=90] to (2.75,0)[out=-90,in=0] to (south2.center);
	}
\caption{$(\calT_+^{i,j})'$ (left) and $\widetilde{\calT}^{\tiny{i,j}}_{+}$ (right).}
\label{fig:mT}	
\end{figure} 

\begin{figure}[h]
\tikzmath{\x1=-4; \x2=-1;\x3=3;}
\begin{tikzpicture}
%%%%%%%%%%%%%%%%%%
%%% poles
  \node (north) at (0,5.5){};
  \node (south) at (0,-5){};
%%%%%%%%%%%%%%%%%%
%T1 chain:
  \node[regular polygon,regular polygon sides=6, draw=black, inner sep=15, fill=white] (T11) at (\x1,3){};
  \node at (T11) {\rotatebox{180}{$(\calT^{\,1,1}_{+})'$}};
  \filldraw (\x1,2) circle (2pt);
  \filldraw (\x1,1) circle (2pt);
  \node (T1dots) at (\x1,1.5){$\vdots$};
  \filldraw (\x1,-1) circle (2pt);
  \node[regular polygon,regular polygon sides=6, draw=black, inner sep=15, fill=white] (T1m) at (\x1,0) {};
  \node at (T1m) {\rotatebox{180}{$(\calT^{\,1,m_1}_{+})'$}};

  \node (w1start) at (\x1,-1){};
  \filldraw (w1start) circle (2pt);
  \filldraw (\x1,-1.5) circle (2pt);
  \node[anchor=center] (w1dots) at (\x1,-2.25) {$\vdots$};
  \filldraw (\x1,-3) circle (2pt);
  \node (w1end) at (\x1,-3.5){};
  \filldraw (w1end) circle (2pt);

  \draw[thick] (north.center) -- (T11.north);
  \draw[thick] (T11.south) -- (T1dots) -- (T1m.north);
  \draw[thick] (T1m.south) -- (w1dots) -- (w1end.center) -- (south.center);

%%%%%%%%%%%%%%%%%%
%Tn2 chain:
  \node[regular polygon,regular polygon sides=6, draw=black, inner sep=15, fill=white] (Tn1) at (\x2,3){};
  \node at (Tn1) {\rotatebox{180}
  {\text{\scalebox{.8}
   {$(\calT^{\text{\scalebox{.75}{\, $n-1,1$}}}_{+})'$}
  }}
  };
  \filldraw (\x2,2) circle (2pt);
  \filldraw (\x2,1) circle (2pt);
  \node (Tndots) at (\x2,1.5){$\vdots$};
  \filldraw (\x2,-1) circle (2pt);
  \node[regular polygon,regular polygon sides=6, draw=black, inner sep=16, fill=white] (Tnm) at (\x2,0){};
  \node at (Tnm) {\rotatebox{180}
  	{\text{\scalebox{.8}
  	{$
  		(\calT
  		^{\text{\scalebox{.75}{$\, n-1,m_{n-1}$}}}
  		_{+})'
  $}}}};

  \node (wnstart) at (\x2,-1){};
  \filldraw (wnstart) circle (2pt);
  \filldraw (\x2,-1.5) circle (2pt);
  \node[anchor=center] (wndots) at (\x2,-2.25) {$\vdots$};
  \filldraw (\x2,-3) circle (2pt);
  \node (wnend) at (\x2,-3.5){};
  \filldraw (wnend) circle (2pt);

  \draw[thick] (north.center) -- (Tn1.north);
  \draw[thick] (Tn1.south) -- (Tndots) -- (Tnm.north);
  \draw[thick] (Tnm.south) -- (wndots) -- (wnend.center) -- (south.center);

%%%%%%%%%%%%%%%%
%%% ldots
\node[anchor = center] at ({(\x1+\x2)/2},3) {$\cdots$};
\node[anchor = center] at ({(\x1+\x2)/2},0) {$\cdots$};

%Tn3 chain:
  \node[regular polygon,regular polygon sides=6, draw=black, inner sep=15, fill=white] (Tn1) at (\x3,3){};
  \node at (Tn1) {\rotatebox{180}{$\widetilde {\calT}^{\, n,1}_{+}$}};
  \filldraw (\x3,2) circle (2pt);
  \filldraw (\x3,1) circle (2pt);
  \node (Tndots) at (\x3,1.5){$\vdots $};
  \filldraw (\x3,-1) circle (2pt);
  \node[regular polygon,regular polygon sides=6, draw=black, inner sep=15, fill=white] (Tnm) at (\x3,0){};
  \node at (Tnm) {\rotatebox{180}{$\widetilde \calT^{\, n,m_n}_{+}$}};

  \node (wnstart) at (\x3,-1){};
  \filldraw (wnstart) circle (2pt);
  \filldraw (\x3,-1.5) circle (2pt);
  \node[anchor=center] (wndots) at (\x3,-2.25) {$\vdots$};
  \filldraw (\x3,-3) circle (2pt);
  \node (wnend) at (\x3,-3.5){};
  \filldraw (wnend) circle (2pt);

  \draw[thick] (north.center) -- (Tn1.north);
  \draw[thick] (Tn1.south) -- (Tndots) -- (Tnm.north);
  \draw[thick] (Tnm.south) -- (wndots) -- (wnend.center) -- (south.center);

%%%%%%%%%%%%%%%%%
% additional edge
\draw[thick]
  (north.center) to[out=-70, in=90]
  (1,0) to[out=-90,in=70]
  (south.center);
%%bows
\draw[thick]
  (north.center) to[out=-45, in=180]
  (\x3,2) to[out=180,in=90] ({\x3-.5},1.75);
\draw[thick]
  ({\x3-.5},1.25) to[out=-90, in=180]
  (\x3,1) to[out=180,in=90]
  (\x3-1.1,0) to[out=-90 ,in=180]
  (\x3,-1);

  \filldraw[fill=white] (north) circle (3pt) node[scale=1] {$\star$};
  \filldraw[fill=white] (south) circle (3pt) node {$\star$};
  
\end{tikzpicture}
\caption{The reduced Tait graph $\calT_+'$ of the tangle for $\Gamma$.}
\label{fig:final}	
\end{figure}
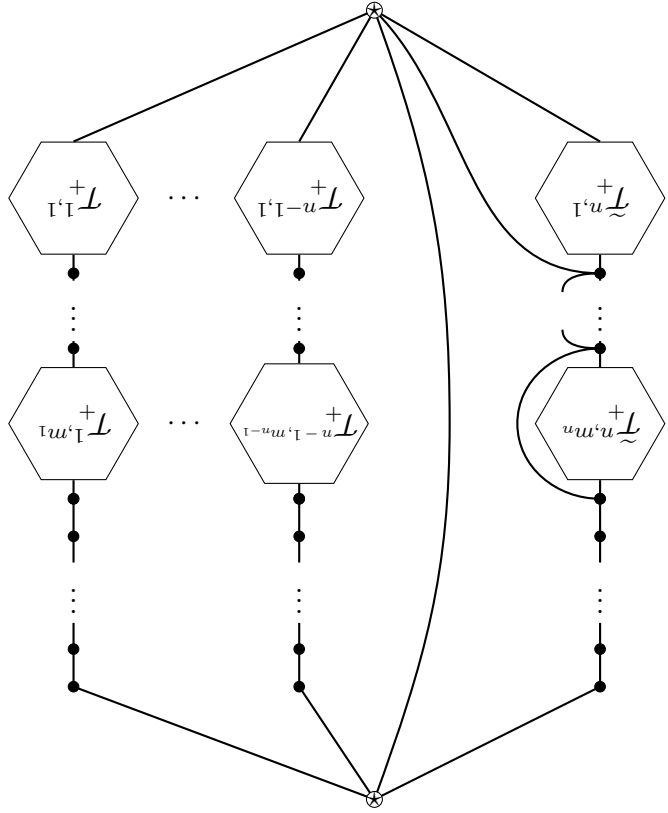
\noindent to the bottom vertex in Figure \ref{fig:final} separates edge-connected polygons of sizes
	\bean
	\{w(w_n)+m_n +1 \}
	\;\cup\;
	\bigcup_{j=1}^{m_n} \{w(v)+e(v): v\in(\calV_{n,j})_+\}
	\eean
	counted with multiplicities on the right
	from the rest.
	If we apply the same procedure to the remaining strings on the left, we obtain that
	$\calT_+'$ consists of edge-connected polygons of sizes
	\bean
	\bigcup_{i=1}^{n}
	\bigcup_{j=1}^{m_i} \{w(v)+e(v): v\in(\calV_{i,j})_+\}
	\cup
	\bigcup_{i=1}^n \{w(w_i)+m_i +1\}
	\=
	\{w(v)+e(v): v\in\calV_{+}\}
	\eean
	since $m_i +1=e(w_i)$ counted with multiplicities.
	\end{proof}

\begin{proposition}\label{prp:taitgarph_arbo_zeroless}
Let $\Gamma$ and $L$ be as in Theorem~\ref{main} and assume that $|\calV|\geq 2$. Then the reduced Tait graph $\calT_+'$ of $L$ consists of edge-connected polygons $(P_v)_{v\in\calV_+}$ of sizes $w(v) + e(v)$ for $v\in\calV_+$.
\end{proposition}

\begin{proof}
Choose $v_0\in\calV_-$ and consider the alternating weighted rooted tree $\Gamma$ with root $v_0$. By Lemma~\ref{lma:edgecon}, the reduced $+$-Tait graph for $\Gamma$ consists of edge-connected polygons of sizes $w(v) + e(v)$ for $v\in\calV_+$ which agrees with the Tait graph $\calT_+$ of $L$ after unmarking the marked vertices, cf.  Remark~\ref{T2T}.
\end{proof}

For our discussion in Section 5, we prove the following result. Here, we assume that $\Gamma$ has no leaf of weight $\pm 1$. Note that all leaves of weight $\pm 1$ can be removed according to \cite[Section~17.3 (1.1)]{afhkln}.

\begin{corollary}\label{cor:0inwVTait}
Let $\Gamma$ be an alternating weighted tree with arborescent link $L$ and reduced Tait graph $\calT_+'$ of $L$ such that $\Gamma$ has no leaf of weight $\pm 1$. Then $\calT_+'$ is not the edge-connected sum of polygons if and only if $0\in w(\calV_-)$.
\end{corollary}
\begin{proof}
If $0\notin w(\calV_-)$, then Proposition~\ref{prp:taitgarph_arbo_zeroless} implies that $\calT_+'$ consists of edge-connected polygons. Conversely, if $v_0\in\calV_-$ with $w(v_0)=0$, then Lemma~\ref{lma:pf1} implies that the reduced Tait graph $\calT_+'$ of $\Gamma$ with root $v_0$ is given as in Figure~\ref{fig:rT} without an edge from the top vertex to bottom vertex. There is no such edge because $w(v_0)=0$ and no edge from an inner subgraphs because $\Gamma$ has no leaf of weight $\pm 1$, cf. Remark~\ref{leafcase}. By Remark~\ref{reduced} (ii), the degree of $v_0$ is greater than $2$. Thus, $\calT_+'$ is not the edge-connected sum of polygons because none of the edges drawn in Figure~\ref{fig:rT} decomposes $\calT_+'$ into polygons.
\end{proof}

We are now in a position to prove our main result.

\begin{proof}[Proof of Theorem \ref{main}]
Let $m \in \mathbb{N}$. Let $L$ be an alternating link such that its reduced Tait graph $\calT_{+}'$ consists of $m$ edge-connected
polygons of sizes $b_i\geq 2$, $i=1,\ldots,m$. According to \cite[Theorem~2]{ad1}, $\Phi_L(q)$ is uniquely determined by $\mathcal T'_+$. From \cite[Theorem~5.1]{ad1} in combination with \cite[Theorems~3.6 and 3.8]{ad1}, we have 
 \bea \label{key}
\Phi_L(q) \= \prod_{i=1}^m h_{b_i}.
\eea
The result now follows from Proposition \ref{prp:taitgarph_arbo_zeroless} and (\ref{key}).
\end{proof}

\section{Applications of Theorem \ref{main}}
In this section, we give some consequences of Theorem~\ref{main}. Since arborescent links are the same as algebraic links, the subsequent results are presented using Conway's notation \cite{ac}.

\subsection{2--bridge knots}
A \emph{2--bridge knot} (or {\it rational knot}) $K$ can be constructed from a weighted tree given in Figure \ref{fig:2b}
\begin{figure}[ht]
\begin{tikzpicture}
\draw[thick] (0,0) -- (3.5,0);
\filldraw [black] (0,0) circle (2pt) node[anchor=south] {$-d_1$};
\filldraw [black] (1.5,0) circle (2pt) node[anchor=south] {$d_{2}$};
\filldraw [black] (3,0) circle (2pt) node[anchor=south] {$-d_{3}$};
\node at (4,0) {$\ldots$};
\draw[thick] (4.5,0) -- (5,0);
\filldraw [black] (5,0) circle (2pt) node[anchor=south] {$(-1)^{n} d_{n}$};
\end{tikzpicture}
\caption{A weighted tree for $2$-bridge knots.}
\label{fig:2b}	
\end{figure}
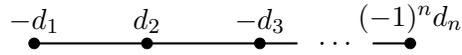
for $d_1,\ldots, d_n\in\Z_{>0}$, $n\geq 1$. The Conway notation for $K$ is $[d_n \dots d_1]$. Here, the vertices $v_j$ have weights $w(v_j)=(-1)^j d_j$. We partition the vertices into the sets \hbox{$\calV_+ = \{v_j: j \text{ even}\}$} and \hbox{$\calV_-=\{v_j: j \text{ odd}\}$}. In the case $n\geq 2$, we have $e(v_j) = 1$ if $j=1$ or $n$ and $e(v_j)=2$ otherwise. In the case $n=1$, we have $e(v_1)=0$.
With the vector 
\bean
 (b_1, \dotsc, b_n)
 \coloneqq
 (d_1 + e(v_1), \dotsc, d_n + e(v_n)),
\eean
Theorem~\ref{main} immediately implies the following result.
\begin{corollary} \label{2b}
If $K$ is a 2--bridge knot as above, then
\begin{equation*} \label{eq:rationaltails}
\Phi_K(q) \,= \prod_{\substack{j=1 \\ j \text{ even}}}^{n}
h_{b_j},\qquad\quad
\Phi_{K^*}(q) \,= \prod_{\substack{j=1 \\  j \text{ odd}}}^{n} h_{b_j}.
\end{equation*}
\end{corollary}
We recall that our convention from Remark~\ref{reduced} (ii) excludes the case $n=1$, $|d_1|\leq 1$ which would give the unknot.

\begin{example}
(1) In Example~\ref{ex:arborescent_knots} (1), the knot $K=5_2$ was constructed from the weighted tree given in Figure \ref{fig:w52}. Thus, the Conway notation is $[3\,2]$ and so by Corollary \ref{2b}, we have (\ref{eq:tail52}). As a comparison, (\ref{eq:tail52}) also follows from (\ref{key}) and the Tait graphs of $K=5_2$ given in Figure~\ref{fig:dTrT}.

\smallskip

\noindent (2) Let $K=7_4$ which has Conway notation $[3\,1\,3]$. An associated weighted tree is given by Figure \ref{fig:w74} and so by Corollary \ref{2b}, we have
\bea\label{eq:tail74}
\Phi_{7_4}(q) \,=\, h_3,\qquad\quad
\Phi_{7_4^{*}}(q) \,=\, h_4^2.
\eea

\begin{figure}[ht]
\begin{equation*}
\begin{tikzpicture}
\draw[thick] (0,0) -- (2,0);
\filldraw (0,0) circle (2pt) node[anchor=south] {$-3$};
\filldraw (1,0) circle (2pt) node[anchor=south] {$1$};
\filldraw (2,0) circle (2pt) node[anchor=south] {$-3$};
\end{tikzpicture}
\end{equation*}
\caption{A weighted tree for $7_4$.}
\label{fig:w74}	
\end{figure}
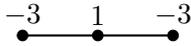
\noindent Alternatively, the reduced Tait graphs $\mathcal T_{\pm}^{\prime}$ of $K=7_4$ are given in Figure \ref{fig:rT76} and so (\ref{eq:tail74}) also follows from~(\ref{key}). 

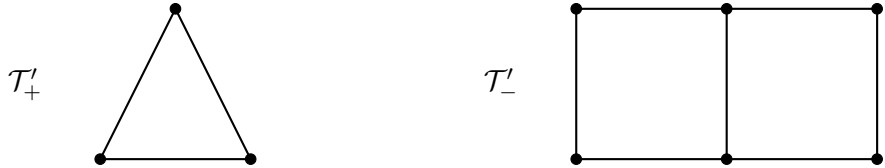
\begin{figure}[ht]
\begin{minipage}{.45\textwidth}
\begin{equation*}
\begin{tikzpicture}[scale=1]
    \node at(-2,0){$\mathcal T_{+}^{\prime}$};
    \filldraw[black] (0,1) circle (2pt) node (a){};
    \filldraw [black] (1,-1) circle (2pt) node (b){};
    \filldraw [black] (-1,-1) circle (2pt) node (d){};
    \draw[thick] (a.center) -- (b.center);
    \draw[thick] (d.center) -- (a.center);
    \draw[thick] (d.center) -- (b.center);
\end{tikzpicture}
\end{equation*}
\end{minipage}
\begin{minipage}{.45\textwidth}
\begin{equation*}
\begin{tikzpicture}[scale=1]
\node at(-3,0){$\mathcal T_{-}^{\prime}$};
\filldraw[black] (0,1) circle (2pt) node (a){};
\filldraw [black] (0,-1) circle (2pt) node (b){};
\filldraw [black] (-2,1) circle (2pt) node (c){};
\filldraw [black] (2,1) circle (2pt) node (d){};
\filldraw [black] (2,-1) circle (2pt) node (e){};
\filldraw [black] (-2,-1) circle (2pt) node (f){};
\draw[thick] (a.center) -- (b.center);
\draw[thick] (f.center) -- (b.center);
\draw[thick] (c.center) -- (a.center);
\draw[thick] (d.center) -- (a.center);
\draw[thick] (e.center) -- (b.center);
\draw[thick] (d.center) -- (e.center);
\draw[thick] (f.center) -- (c.center);
\end{tikzpicture}
\end{equation*}
\end{minipage}
\caption{The reduced Tait graphs for $7_4$.}
\label{fig:rT76}
\end{figure}
\end{example}

\subsection{Montesinos knots}
A \emph{Montesinos knot} $K$ can be constructed by a star-shaped weighted tree \cite[Section 17.6.2]{afhkln} as in Figure \ref{fig:M} with center of weight $-k$, $k\in\Z_{\geq 0}$, $m\in\Z_{\geq 1}$ rays of length $n_i\in\Z_{\geq 1}$ and $d^{(i)}_{1},
  d^{(i)}_{2}, 
  \cdots
  \,d^{(i)}_{n_i} 
  \in\Z_{\geq 1}$ for $i=1,\ldots,m$. Note that such a tree is alternating. Moreover, if $k=0$, we assume $m\geq 3$, cf. Remark~\ref{reduced} (ii).
  The Conway notation for $K$ is
\bean[]
[\mathbf d^{(1)};\,\mathbf d^{(2)};\ldots;\,\mathbf d^{(m)} +^k]
\eean
where $+^k$ means $k$ copies of $+$ and $\mathbf d^{(i)}$ denotes the concatenation of the entries in the $i$th ray. For $v\in\mathcal V$, we have
\bean
  e(v) \= \begin{cases}
    m &\text{if $v$ is the center,}\\
    1 &\text{if $v$ is a leaf,}\\
    2 &\text{otherwise.}
  \end{cases}
 \eean

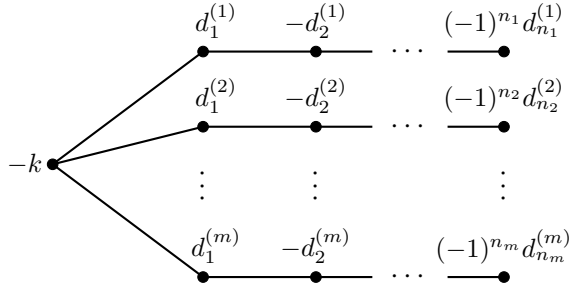
\begin{figure}
\begin{equation*}
\tikzmath{\llll = 0.05;}
\begin{tikzpicture}
\node (k) at (-.5,0){};
\filldraw [black] (k) circle (2pt) node[anchor=east] {\small$-k$};

\draw[thick] (k.center) -- (1.5,1.5) -- (3.75,1.5);
\filldraw [black] (1.5,1.5) circle (2pt)
  node[anchor=south] at(1.5,1.5+\llll)
  {\small $\quad d_{1}^{(1)}$};
\filldraw [black] (3,1.5) circle (2pt)
  node[anchor=south] at(3,1.5+\llll)
  {\small$-d_{2}^{(1)}$};
\node at (4.25,1.5) {$\cdots$};
\filldraw [black] (5.5,1.5) circle (2pt)
  node[anchor=south] at (5.5,1.5+\llll)
  {\small$(-1)^{n_1 + 1} d_{n_1}^{(1)}$};
\draw[thick] (4.75,1.5) -- (5.5,1.5);

\draw[thick] (k.center) -- (1.5,.5)--(3.75,0.5);
\filldraw [black] (1.5,0.5) circle (2pt)
  node[anchor=south] at (1.5,0.5+\llll)
  {\small $\quad d_{1}^{(2)}$};
\filldraw [black] (3,0.5) circle (2pt)
  node[anchor=south] at (3,0.5+\llll)
  {\small$-d_{2}^{(2)}$};
\node at (4.25,0.5) {$\cdots$};
\filldraw [black] (5.5,0.5) circle (2pt)
  node[anchor=south] at (5.5,0.5+\llll)
  {\small$(-1)^{n_2 + 1} d_{n_2}^{(2)}$};
\draw[thick] (4.75,0.5) -- (5.5,0.5);

\draw[thick] (k.center) -- (1.5,-1.5) -- (3.75,-1.5);
\filldraw [black] (1.5,-1.5) circle (2pt)
  node[anchor=south] at (1.5,-1.5+\llll)
  {\small $\quad d_{1}^{(m)}$};
\filldraw [black] (3,-1.5) circle (2pt)
  node[anchor=south] at (3,-1.5+\llll)
  {\small$-d_{2}^{(m)}$};
\node at (4.25,-1.5) {$\cdots$};
\filldraw [black] (5.5,-1.5) circle (2pt)
  node[anchor=south] at(5.5,-1.5+\llll)
  {\small$(-1)^{n_m + 1} d_{n_m}^{(m)}$};
\draw[thick] (4.75,-1.5) -- (5.5,-1.5);

\node at (5.5,-.25) {{\rotatebox{90}{$\cdots$}}};
\node at (1.5,-.25) {{\rotatebox{90}{$\cdots$}}};
\node at (3,-.25) {{\rotatebox{90}{$\cdots$}}};

\end{tikzpicture}
\end{equation*}
\caption{A weighted tree for Montesinos knots.}
\label{fig:M}
\end{figure}
\noindent For $i=1,\ldots,m$ and $j=1, \dotsc, n_{i}$, define
\begin{equation*}
b_j^{(i)} = d^{(i)}_{j} + e(v_{i, j})
\end{equation*}
where $v_{i, j}$ is the $j$th vertex in the $i$th ray. Another application of Theorem \ref{main} is the following.
\begin{corollary} \label{mk} 
If $K$ is a Montesinos knot as above, then 
  \bean
    \Phi_{K^*}(q) \= 
    h_{m+k} \prod_{i=1}^m\ 
    \prod_{\substack{j=1\\j \text{ even}}}^{n_i}}
    h_{b_j^{(i)}
  \eean
  and, if $k\neq 0$,
  \bean
    \Phi_{K}(q) \= 
    \prod_{i=1}^m\ 
    \prod_{\substack{j=1 \\ j\text{ odd}}}^{n_i}}
    h_{b_j^{(i)}.
  \eean
\end{corollary}

\begin{example}\label{ex:tails}
Consider the Montesinos knot $K=9_{16}$ which has Conway notation $[3;3;2+]$. A weighted tree and diagram for $9_{16}$ are given in Figure \ref{fig:916}. By Corollary \ref{mk}, we have
\bea \label{tail916}
\Phi_{9_{16}}(q) \= h_4^2 h_3,\qquad
\Phi_{9_{16}^*}(q) \= h_4.
\eea
Alternatively, the reduced Tait graphs of $9_{16}$ are given in Figure \ref{fig:T916} and so (\ref{tail916}) also follows from (\ref{key}).

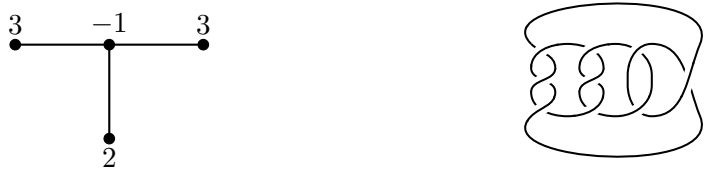
\begin{figure}[ht]
\begin{equation*}
\tikzmath{\sc = 1.25;} 
\begin{tikzpicture}
\draw[thick] (-\sc,0) -- (\sc,0);
\draw[thick] (0,0) -- (0,-\sc);
\filldraw (-\sc,0) circle (2pt) node[anchor=south] {$3$};
\filldraw (0,0) circle (2pt) node[anchor=south] {$-1$};
\filldraw (\sc,0) circle (2pt) node[anchor=south] {$3$};
\filldraw (0,-\sc) circle (2pt) node[anchor=north] {$2$};
\end{tikzpicture}
\hspace{100pt}
\includegraphics[width=.2\textwidth]{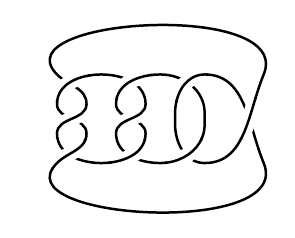}
\end{equation*}
\caption{A weighted tree and diagram for $9_{16}$.}
\label{fig:916}
\end{figure}

\begin{figure}[ht]
\begin{minipage}{.45\textwidth}
\begin{equation*}
\begin{tikzpicture}[scale=1]
    \node at(-2,0){$\mathcal T_+'$};
    \filldraw[black] (0,3/2) circle (2pt) node (a){};
    \filldraw [black] (0,-3/2) circle (2pt) node (b){};
    \filldraw [black] (-1,1/2) circle (2pt) node (c){};
    \filldraw [black] (-1,-1/2) circle (2pt) node (c2){};
    \filldraw [black] (1,0) circle (2pt) node (d){};
    \filldraw [black] (0,1/2) circle (2pt) node (e){};
    \filldraw [black] (0,-1/2) circle (2pt) node (e2){};

    \draw[thick] (a.center) -- (b.center);
    \draw[thick] (c.center) -- (c2.center);
    \draw[thick] (c.center) -- (a.center);
    \draw[thick] (c2.center) -- (b.center);
    \draw[thick] (a.center) -- (e.center) -- (e2.center) -- (b.center);
    \draw[thick] (a.center) -- (d.center) -- (b.center);
    \draw[thick] (a.center)
      to[out=0,in=90] (2,0)
      to[out=-90,in=0] (b.center);
\end{tikzpicture}
\end{equation*}
\end{minipage}
\begin{minipage}{.45\textwidth}
\begin{equation*}
\begin{tikzpicture}[scale=1]
    \node at(-1,2/2){$\mathcal T_-'$};
    \filldraw[black] (0,0) circle (2pt) node (a){};
    \filldraw [black] (2,0) circle (2pt) node (b){};
    \filldraw [black] (0,2) circle (2pt) node (c){};
    \filldraw [black] (2,2) circle (2pt) node (d){};
    \draw[thick] (a.center) -- (b.center);
    \draw[thick] (c.center) -- (d.center);
    \draw[thick] (a.center) -- (c.center);
    \draw[thick] (b.center) -- (d.center);
\end{tikzpicture}
\end{equation*}
\end{minipage}
\caption{The reduced Tait graphs for $9_{16}$.}
\label{fig:T916}
\end{figure}
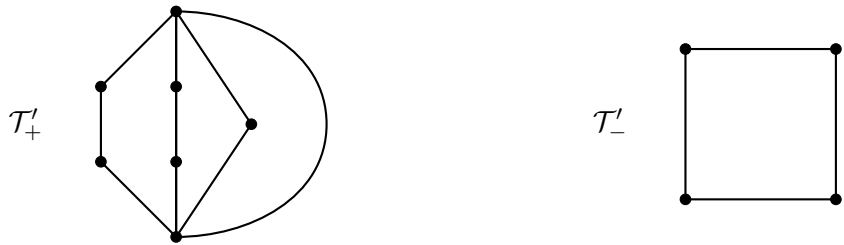
\end{example}

\subsection{A non-Montesinos knot}
Lastly, we consider the knot $K=11a_{250}$ which has Conway notation \hbox{$[(3;2)1(3;2)]$} and is not a Montesinos knot \cite{cmo}. A weighted tree \cite[p.~38]{ac} and diagram for $11a_{250}$ are given in Figure \ref{fig:11a250}. Hence, given the presentation in Figure \ref{fig:11a250}, Theorem~\ref{main} is only applicable to $11a_{250}^*$. Thus, we have
\bea \label{tail11a250}
\Phi_{11a_{250}^*}(q) \= h_3^2.
\eea
In contrast, the reduced Tait graphs of $11a_{250}$ are given in Figure \ref{fig:T11a250} and so (\ref{tail11a250}) also follows from (\ref{key}).
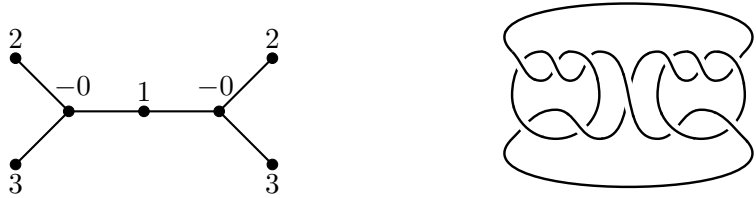
\begin{figure}[ht]
\begin{equation*}
\begin{tikzpicture}
\draw[thick] (-1,0) -- (1,0);
\draw[thick] (-1,0) -- ({-1-1/sqrt(2)},{1/sqrt(2)});
\draw[thick] (-1,0) -- ({-1-1/sqrt(2)},-{1/sqrt(2)});
\draw[thick] (1,0) -- ({1+1/sqrt(2)},{1/sqrt(2)});
\draw[thick] (1,0) -- ({1+1/sqrt(2)},-{1/sqrt(2)});
\filldraw [] ({-1-1/sqrt(2)},{1/sqrt(2)}) circle (2pt) 
  node[anchor=south] {$2$};
\filldraw [] ({-1-1/sqrt(2)},-{1/sqrt(2)}) circle (2pt) 
  node[anchor=north] {$3$};
\filldraw [black] (-1,0) circle (2pt) 
  node[anchor=south] at(-1+.05,.05) {$-0$};
\filldraw [] (0,0) circle (2pt)
  node[anchor=south] {$1$};
\filldraw [black] (1,0) circle (2pt)
  node[anchor=south ] at(1-.05,.05) {$-0$};
\filldraw [] ({1+1/sqrt(2)},{1/sqrt(2)}) circle (2pt)
  node[anchor=south] {$2$};
\filldraw [] ({1+1/sqrt(2)},-{1/sqrt(2)}) circle (2pt) 
  node[anchor=north] {$3$};
\node at (0,-1){};
\end{tikzpicture}
\hspace{60pt}
\includegraphics[width=.3\textwidth]{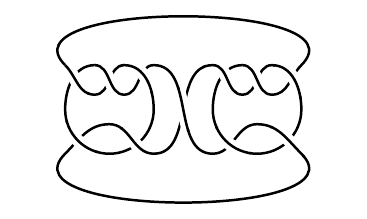}
\end{equation*}
\caption{A weighted tree and diagram for $11a_{250}$.}
\label{fig:11a250}
\end{figure}

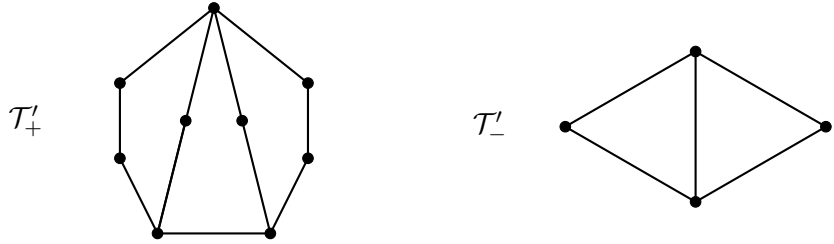
\begin{figure}[ht]
\begin{minipage}{.4\textwidth}
\begin{equation*}
\begin{tikzpicture}[scale=1]
    \node at(-2.5,0){$\mathcal T_+'$};
    \filldraw[black] (0,3/2) circle (2pt) node (a){};
    \filldraw [black] (-3/4,-3/2) circle (2pt) node (b){};
    \filldraw [black] (3/4,-3/2) circle (2pt) node (b2){};
    \filldraw [black] (-5/4,1/2) circle (2pt) node (c){};
    \filldraw [black] (-5/4,-1/2) circle (2pt) node (c2){};
    \filldraw [black] (-3/8,0) circle (2pt) node (d){};
    \filldraw [black] (3/8,0) circle (2pt) node (d2){};
    \filldraw [black] (5/4,1/2) circle (2pt) node (e){};
    \filldraw [black] (5/4,-1/2) circle (2pt) node (e2){};

    \draw[thick] (a.center) -- (d.center) -- (b.center);
    \draw[thick] (a.center) -- (d2.center) -- (b2.center);
    \draw[thick] (a.center) --(c.center) -- (c2.center)-- (b.center);
    \draw[thick] (a.center) -- (e.center) -- (e2.center) -- (b2.center) -- (b.center);
    \draw[thick] (d.center) -- (b.center);
\end{tikzpicture}
\end{equation*}
\end{minipage}
\begin{minipage}{.4\textwidth}
\begin{equation*}
\begin{tikzpicture}[scale=1]
    \node at(-1,0){$\mathcal T_-'$};
    \filldraw[black] (0,0) circle (2pt) node (d){};
    \filldraw[black] ({2*sqrt(3)},0) circle (2pt) node (e){};
    \filldraw [black] ({sqrt(3)},1) circle (2pt)node (a){};
    \filldraw [black] ({sqrt(3)},-1) circle (2pt)node (c){};

    \draw[thick] (a.center) -- (c.center);
    \draw[thick] (a.center) -- (d.center);
    \draw[thick] (c.center) -- (d.center);
    \draw[thick] (a.center) -- (e.center);
    \draw[thick] (c.center) -- (e.center);
\end{tikzpicture}
\end{equation*}
\end{minipage}
\caption{The reduced Tait graphs for $11a_{250}$.}
\label{fig:T11a250}
\end{figure}

\section{Asymptotics of $\Phi_{K}(q)$}

The first arborescent knot for which Theorem \ref{main} is not applicable is $8_5$. This knot is the first case in a family of pretzel knots for which Theorem \ref{main} does not apply. The first non-arborescent knot is $8_{18}$.  In this section, we explain why the tail in these situations can not be written as a product of the functions $h_b=h_b(q)$ and provide numerical evidence that, more generally, these tails are not classical modular forms, quasimodular forms or mock modular forms. This leads to the question of the classification of alternating knots $K$ such that $\Phi_K(q)$ can be written as a product of $h_b$'s, see Question~\ref{q} and Remark~\ref{rq}.

We first compare the asymptotics of $h_b(e^{-h})$ as $h\to0$ with the asymptotics of $\Phi_K(e^{-h})$ for a given knot $K$. This is motivated by a similar approach for the classification of modular Nahm sums \cite{cgz,nahm,thesis,vz,zdilog}. As $h\to0$ on a ray in the right half-plane, we have
\bea\label{eq:asymphb}
	h_b(e^{-h}) \=
	\begin{cases}
		2e^{-\pi^2/2bh} \sqrt{\frac{2\pi}{bh}}\,
		\sin \left ( \tfrac \pi {b} \right) \left(1 +\text O(h) \right)
		&\text{if $b \geq 3$ is odd,}\\
		\frac2b
		+ \text O(h)
		&\text{if $b$ is even}.
	\end{cases}\\
\eea
The computation (\ref{eq:asymphb}) follows from the usual modular transformation of $h_b$ if $b$ is odd and from \cite[Proposition, p.~98]{lz} if $b$ is even. In both cases, we have $\lim_{h\to0} h\log(h_b(e^{-h})) \in\pi^2\Q$. Hence, if a $q$-series $f$ with asymptotics $h\log(f(e^{-h})) \to V$ as $h\rightarrow 0$ for some $V\in\C$ can be written as a product of $h_b$'s, the asymptotics in~\eqref{eq:asymphb} imply that $V\in\pi^2 \Q$. Although this condition cannot be verified numerically, computations can suggest if $V\in\pi^2 \Q$ is likely.

\begin{remark} \label{mod}
A similar argument can be used to (numerically) exclude other modular behavior. If $g$ is a modular form of integer (or half-integer) weight for some subgroup of $\text{SL}_2(\Z)$ of finite index, then its modular transformation implies that
\bea\label{eq:gasymp}
h\log(g(e^{-h})) \to V
\eea
as $h\to0$ on a ray in the right half-plane
for some $V\in\pi^2\Q$,
see \cite[Lemma~3.1]{vz}.
The same proof is also applicable to quasimodular forms (and mock modular forms) because they (their completions) have a similar transformation under $S=\bigl(\begin{smallmatrix}
 0 & -1 \\ 1 & 0
\end{smallmatrix}\bigr)$ (for details, see \cite{af, kz,z}).
Therefore, if a $q$-series $f$ has an asymptotic as in~\eqref{eq:gasymp}, $V\notin \pi^2\Q$ would also exclude these modular behaviors of $q^cf(q)$ for any $c\in\Q$.
\end{remark}

\begin{remark} \label{hb}
From (\ref{hbdef}) and (\ref{epsilon}), we have for $b \geq 3$
\begin{equation} \label{hbexpansion}
h_{b}(q) = 1 - q + (-1)^b q^{b-1} + \text O(q^{b+2}). 
\end{equation}
Thus, a product of $h_b$'s with $m$ non-trivial factors, including $r_3$ copies of $h_3$ has coefficients $-m$ at $q$ and $\binom{m}{2} - r_3$ at $q^2$.
\end{remark}

\subsection{Pretzel knots} \label{pk}
A pretzel knot $P(d_1,\ldots,d_n)$ with integers $d_i \geq 2$ for $i=1, \dotsc, n$ is an arborescent knot associated with a star-shaped weighted tree, i.e., a Montesinos knot with center~$0$ and~$n$ rays of length one where the leaves have weight $d_1,\ldots,d_n$, see Figure \ref{fig:pretzel}. Here, Theorem~\ref{main} is not applicable. The tail of the colored Jones polynomial for pretzel knots of the form $K=P(2k+1,2u+1,2)$ for $k,u\in\Z_{\geq1}$ has been computed explicitly. 

\begin{figure}[H]
\begin{equation*}
\tikzmath{\llll = 0.05;}
\begin{tikzpicture}
\node (k) at (0,0){};
\filldraw [black] (k) circle (2pt) node[anchor=east] {$-0$};

\draw[thick] (k.center) -- (1.5,1.5);
\filldraw [black] (1.5,1.5) circle (2pt)
  node[anchor=west]{$d_1$};

\draw[thick] (k.center) -- (1.5,.5);
\filldraw [black] (1.5,0.5) circle (2pt)
  node[anchor=west] {$d_2$};

\draw[thick] (k.center) -- (1.5,-1.5);
\filldraw [black] (1.5,-1.5) circle (2pt)
  node[anchor=west] {$d_n$};

\node at (1.5,-.5) {{\rotatebox{90}{$\cdots$}}};

\end{tikzpicture}
\end{equation*}
\caption{A weighted tree for $P(d_1,\ldots,d_n)$.}
\label{fig:pretzel}
\end{figure}
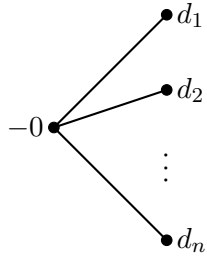

\noindent Recall that the $q$-binomial coefficient is given by
\bean
\qbinom nk
\coloneqq
\frac{(q)_n}{(q)_k(q)_{n-k}}\thin.
\eean
By \cite[Theorem~3.1]{eh1}, we have
%\begin{theorem}[\cite{eh1}, Theorem 3.1] \label{thm:pretzeltail}
%For $K =P(2k+1,2u+1,2)$ with $k,u\in\Z_{\geq 1}$, we have
\bea\label{eq:tailpretzel}
\Phi_K(q) \thin=\thin 
(q)^2_\infty\,
\sum_{l_1\geq 0} \!\cdots\! \sum_{l_k\geq 0}\ 
\sum_{p_1\geq 0} \!\cdots\! \sum_{p_u\geq 0}\,
\frac{q^{L_1^2 + \cdots +L_k^2 + L_1 + \cdots + L_k}}
{(q)_{l_1}\cdots (q)_{l_k}}\,
\frac{q^{P_1^2 + \cdots +P_u^2 + P_1 + \cdots + P_u}}
{(q)_{p_1}\cdots (q)_{p_u}}\,
\qbinom{l_k + p_u}{p_u}
\eea
where
$L_j = l_j + \cdots + l_k$ for $j=1,\ldots,k$ and
$P_j = p_j + \cdots + p_u$ for $j=1,\ldots,u$.
%\end{theorem}
As discussed in Example~\ref{ex:arborescent_knots} (2), $K=8_5$ can be constructed from the weighted tree given in Figure \ref{fig:w85}. In particular, note that Theorem~\ref{main} is not applicable as $0\in w(\calV_-)$. The Tait graph $\mathcal{T}_{+}$ of $K=8_5$ is given in Figure \ref{fig:T85} and thus (\ref{key}) is also not applicable.

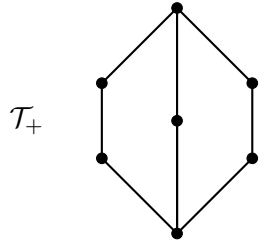
\begin{figure}[ht]
\begin{equation*}
\begin{tikzpicture}
    \node at(-2,0){$\mathcal T_+$};
    \filldraw[black] (0,3/2) circle (2pt) node (a){};
    \filldraw [black] (0,-3/2) circle (2pt) node (b){};
    \filldraw [black] (-1,1/2) circle (2pt) node (c){};
    \filldraw [black] (-1,-1/2) circle (2pt) node (c2){};
    \filldraw [black] (0,0) circle (2pt) node (d){};
    \filldraw [black] (1,1/2) circle (2pt) node (e){};
    \filldraw [black] (1,-1/2) circle (2pt) node (e2){};

    \draw[thick] (a.center) -- (b.center);
    \draw[thick] (c.center) -- (c2.center);
    \draw[thick] (c.center) -- (a.center);
    \draw[thick] (c2.center) -- (b.center);
    \draw[thick] (a.center) -- (e.center) -- (e2.center) -- (b.center);
    \draw[thick] (d.center) -- (b.center);
\end{tikzpicture}
\end{equation*}
\caption{$\mathcal{T}_{+}$ for $8_5$.}
\label{fig:T85}
\end{figure}
\noindent  Observe that $K=8_5=P(3,3,2)$ (see Figure \ref{fig:85sur}) and so~\eqref{eq:tailpretzel} implies
\bea\label{eq:tail85}
    \Phi_{8_5}(q) &\=
    (q)_{\infty}^2
    \sum_{a,b\geq 0}
    \frac{q^{a^2+a+b^2+b}}{(q)_a(q)_b}
    \thin
    \qbinom{a+b}{b} \\
    &\= 1 - 2q + q^2 - 2q^4 + 3q^5 - 3q^8 + q^9 + \operatorname O(q^{10}).
  \eea
By comparing the first two coefficients, we see $\Phi_{8_5}(q) \neq h_4^2h_3$ in contrast to Theorem~\ref{main}. In fact, we can deduce that $\Phi_{8_5}(q)$ is not a finite product of $h_b(q)$'s as follows. Suppose otherwise. Then by Remark \ref{hb}, the first four coefficients in  (\ref{eq:tail85}) force two factors, both of which are $h_5$. But the coefficient of $q^5$ in $h_5^2$ is $2$ which does not match (\ref{eq:tail85}). As $h\searrow0$ and $q=e^{-h} \nearrow 1$, numerics\footnote{Supplementary code is available upon request.} suggest that we have
\bea\label{eq:Phi85asymp1}
	h\log(\Phi_{8_5}(e^{-h}))
	\;\to\;
	V_1
\eea
where, with $X_1\approx0.5436890$ a root of $x^3+x^2+x-1$,
\bean \label{eq:V85}
V_1 & \=  -4\operatorname{Li}_2(X_1)+\Li_2(X_1^2)-2\log(X_1)^2+\frac{\pi^2}6\\
& \= -1.352936859\cdots.
\eean
Here, $\Li_2$ is the dilogarithm function \cite{zdilog} with the principal branching. Although we do not address it here, we note that~\eqref{eq:Phi85asymp1} can potentially be proven using the techniques in \cite{gzasymp,vz,zdilog}.
For example, for $h=.0005$, we have $h\log(\Phi_{8_5}(e^{-h})) \approx  -1.351$. A more precise approximation can be found using extrapolation methods as explained, e.g., in the appendix of \cite{gmz}. Computing $V_1$ to a higher precision, it seems unlikely that $V_1\in\pi^2 \Q$ is true. This suggests that $q^c\Phi_{8_5}(q)$ for any $c\in\Q$ is not modular as in Remark \ref{mod}. If $h\to 0$ on a fixed ray in the right half-plane with $\arg h = .45\pi$, then the limit is even more convincing: using similar asymptotic methods, numerics suggest that we have
\bea\label{eq:Phi85asymp2}
	h\log(\Phi_{8_5}(e^{-h}))
	\;\to\;
	V_2
\eea
where, with $X_2\approx-0.7718445 - 1.115143\thin i$ a root of $x^3+x^2+x-1$,
\bea \label{eq:V285}
	V_2 &\= 
	-4\Li_2(X_2)+\Li_2(X_2^2)-2\log(X_2)^2-4\pi i\log(X_2)+\frac{\pi^2}6\\
	&\= -14.12794\cdots + 3.177293\cdots\thin i,
\eea
and $V_2\notin\pi^2\Q$.

In view of (\ref{key}), the following phenomenon seems to relate the Tait graph $\calT_+$ of $K=8_5$ to the double-sum in~\eqref{eq:tail85}. The Tait graph $\calT_+$ of $K=8_5$ consists of two $5$-gons that are \emph{not} edge-connected. If they were, then $\Phi_{8_5}(q)$ would equal $h_5^2$ by (\ref{key}) and $\Phi_{8_5}(q)$ is \emph{almost} given by $h_5^2$, namely if $\qbinom{a+b}b$ was not present in~\eqref{eq:tail85}, then the resulting $q$-series would equal $h_5^2$ by the second Rogers-Ramanujan identity.

Similar computations have been performed for pretzel knots \hbox{$K=P(2k+1,2u+1,2)$} with $k,u\leq 5$ using \eqref{eq:tailpretzel}. These computations suggest that for all of these pretzel knots we have $h\log(\Phi_K(e^{-h})) \to V$ as $h\to0$ for some %$a\in\C$, $k\in\frac 1 2 \Z$, and 
$V\notin\pi^2\Q$. Therefore, we suspect that $\Phi_K(q)$ for these knots $K$ cannot be written as a product of $h_b$'s.
Moreover, the Tait graph of $K$ consists of two polygons of size $2k+3$ and $2u+3$, respectively. Finally, if $\qbinom {l_k+p_u}{p_u}$ is removed from~\eqref{eq:tailpretzel}, the resulting $q$-series would equal $h_{2k+3}h_{2u+3}$ by the Andrews-Gordon identities \cite{andrews}.

\subsection{The $8_{18}$ knot} \label{818}
We also discuss the modularity of $\Phi_K(q)$ for the first non-arborescent knot $K=8_{18}$ \cite{ac,c}. Although $K=8_{18}$ cannot be constructed from a weighted tree, it does arise from the weighted graph \cite[p.~151]{ac} given in Figure \ref{fig:818}. The Tait graphs of $K = 8_{18}$ are given in Figure \ref{fig:T818} and thus (\ref{key}) is not applicable. Using the algorithm from \cite{glz}, we have (see the Appendix)
\begin{align}\label{eq:tail818}
\Phi_{8_{18}}(q)
\= &(q)_\infty^2 
\sum_{a,b\geq0}(-1)^{a+b}
\frac
{q^{\frac 1 2 a(a+1)+\frac 1 2 b(b+1)}}
{(q)_a(q)_b}
\thin
\qbinom{a+b}{b}
\\\nonumber
\=&
 1 - 4q + 2q^2 + 9q^3 - 5q^4 - 8q^5 - 14q^6 + 10q^7 + 21q^8 + 14q^9 + \text O(q^{10}).
\end{align}

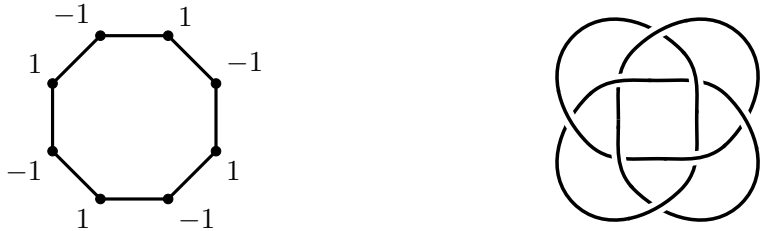
\begin{figure}[ht]
\begin{equation*}
\begin{tikzpicture}[scale=.45]
\node (v1) at (1,2.414214){};
\node (v2) at (2.414214,1){};
\node (v3) at (2.414214,-1){};
\node (v4) at (1,-2.414214){};
\node (v5) at (-1,-2.414214){};
\node (v6) at (-2.414214,-1){};
\node (v7) at (-2.414214,1){};
\node (v8) at (-1,2.414214){};

\filldraw (v1) circle (4pt)
 node[anchor=south west] at (v1){$1$};
\filldraw (v2) circle (4pt)
 node[anchor=south west] at (v2){$-1$};
\filldraw (v3) circle (4pt)
 node[anchor=north west] at (v3){$1$};
\filldraw (v4) circle (4pt)
 node[anchor=north west] at (v4){$-1$};
\filldraw (v5) circle (4pt)
 node[anchor=north east] at (v5){$1$};
\filldraw (v6) circle (4pt)
 node[anchor=north east] at (v6){$-1$};
\filldraw (v7) circle (4pt)
 node[anchor=south east] at (v7){$1$};
\filldraw (v8) circle (4pt)
 node[anchor=south east] at (v8){$-1$};

\draw[very thick] (v1.center) -- (v2.center);
\draw[very thick] (v2.center) -- (v3.center);
\draw[very thick] (v3.center) -- (v4.center);
\draw[very thick] (v4.center) -- (v5.center);
\draw[very thick] (v5.center) -- (v6.center);
\draw[very thick] (v6.center) -- (v7.center);
\draw[very thick] (v7.center) -- (v8.center);
\draw[very thick] (v8.center) -- (v1.center);
\end{tikzpicture}
\hspace{100pt}
\includegraphics[width=0.2\textwidth]{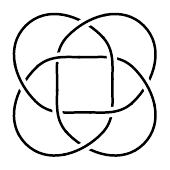}
\end{equation*}
\caption{A weighted graph and diagram for $8_{18}$.}
\label{fig:818}
\end{figure}

\begin{figure}[ht]
\begin{equation*}
\begin{tikzpicture}[scale=1]
\node (v0) at (-2,0){$\mathcal T_\pm$};
\node (v0) at (0,0){};
\node (v1) at (1,1){};
\node (v2) at (1,-1){};
\node (v3) at (-1,-1){};
\node (v4) at (-1,1){};

\filldraw (v1) circle (2pt);
\filldraw (v2) circle (2pt);
\filldraw (v3) circle (2pt);
\filldraw (v4) circle (2pt);
\filldraw (v0) circle (2pt);

\draw[thick] (v1.center) -- (v2.center);
\draw[thick] (v2.center) -- (v3.center);
\draw[thick] (v3.center) -- (v4.center);
\draw[thick] (v4.center) -- (v1.center);
\draw[thick] (v0.center) -- (v1.center);
\draw[thick] (v0.center) -- (v2.center);
\draw[thick] (v0.center) -- (v3.center);
\draw[thick] (v0.center) -- (v4.center);
\end{tikzpicture}
\end{equation*}
\caption{The Tait graphs for $8_{18}$.}
\label{fig:T818}
\end{figure}
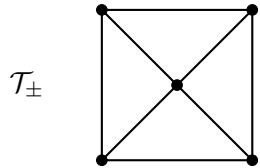
\noindent From the first four coefficients, we see that $\Phi_{8_{18}}(q) \neq h_3^4$ in contrast to Theorem~\ref{main}. Again, using Remark \ref{hb}, one can deduce that $\Phi_{8_{18}}(q)$ is not a finite product of $h_b(q)$'s. The first two coefficients in (\ref{eq:tail818}) force four copies of $h_3$. But the coefficient of $q^3$ in $h_3^4$ is $8$ which does not match (\ref{eq:tail818}). Similar to $8_5$, we can compute the asymptotics of $\Phi_{8_{18}}(q)$ numerically. As $h\searrow0$ on the real axis, it appears that
$h\log(\Phi_{8_{18}}(e^{-h}))\to -\frac{\pi^2}3$
which matches the asymptotics of a modular $q$-series. However, if we instead consider the limit as $h\to 0$ with $\arg h = .45\pi$, then according to numerics
\bea\label{eq:lim818}
h\log(\Phi_{8_{18}}(e^{-h}))
\;\to\;
V_2
\eea
where 
\bea\label{eq:V818}
V_2
&\= -\frac{\pi^2}{2} + 4\Li_2(i)
\= -5.757269\cdots + 3.663862\cdots \thin i.
\eea
This suggests that $\Phi_{8_{18}}(q)$ is not modular as in Remark \ref{mod} as $V_2 \notin \pi^2 \mathbb{Q}$. The Tait graphs $\calT_\pm$ of $K=8_{18}$ in Figure \ref{fig:T818} consist of four triangles and if $\qbinom{a+b}{b}$ was absent in~\eqref{eq:tail818}, then the resulting $q$-series equals $h_3^4 = (q)_\infty^4$ (which would be consistent with (\ref{key}) if it was applicable).

\subsection{Questions and outlook}
The numerical observations in Section~\ref{pk} suggest the following question.
\begin{question} \label{q}
Let $\Gamma = (\mathcal{V}, \mathcal{E}, w)$ be an alternating weighted tree with arborescent knot $K$ as in Theorem~\ref{main} such that $\Gamma$ has no leaf of weight $\pm 1$.
Are the following equivalent:
\begin{enumerate}[label=(\roman*)]
    \item $0\notin w(\calV_-)$,
      \item $\Phi_K(q)$ is a product of $h_b$'s.
\end{enumerate}
\end{question}

\begin{remark}\label{rq}
By Corollary~\ref{cor:0inwVTait}, (i) of Question~\ref{q} is equivalent to the fact that the reduced Tait graph $\mathcal T_+'$ of $K$ is the edge-connected sum of polygons. Moreover, by Theorem~\ref{main} and~(\ref{key}),~(i) implies~(ii).
Clearly, it would be beneficial to consider more examples, e.g., \cite{gsg}.
In view of Section~\ref{818}, one could conjecture the following: for {\it any} alternating knot $K$, (i) and (ii) are equivalent to the fact that $K$ is an arborescent knot as in Theorem~\ref{main} with $0\notin w(\calV_-)$.
\end{remark}

\begin{remark}\label{rvc}
Finally, it would be desirable to understand the topological meaning of the (non)-modularity of $\Phi_K(q)$. Note that the limits in~\eqref{eq:Phi85asymp2} and~\eqref{eq:lim818} for $8_5$ and $8_{18}$, respectively, can be rewritten as
\bea \label{guess}
\lim_{N\to\infty}
\frac{\log (\Phi_K(e^{2\pi i/N}))}N
\= \frac{ iV}{2\pi}
\eea
where $\frac 1 N\to0 $ along a ray in the upper half-plane. Because $\Phi_K(q) = \lim_N q^{c_N}J_N(K,q)$ for some factor $c_N$ that is quadratic in $N$, equation~\eqref{guess} resembles the complexified Volume Conjecture \cite{mmoty, g}. This suggests that $V$ as in~\eqref{eq:V285} and~\eqref{eq:V818} could be related to hyperbolic properties of $K$. This will be investigated in future work.
\end{remark}

\section*{Appendix}
In this section, we describe how to obtain the double-sum representation for the tail of $8_{18}$ as given in \eqref{eq:tail818}. To do so, we first recall the following classical $q$-series identities (see \cite[II.1, II.2, II.4, II.5, II.29]{gr}):
\begin{equation} \label{ii1}
\sum_{n \geq 0} \frac{z^n}{(q)_n} = \frac{1}{(z)_{\infty}},
\end{equation}

\begin{equation} \label{ii2}
\sum_{n \geq 0} \frac{(-1)^n z^n q^{\frac{n(n-1)}{2}}}{(q)_n} = (z)_{\infty},
\end{equation}

\begin{equation} \label{ii4}
\sum_{n=0}^{K} \frac{(-1)^n t^n q^{n(n-1)/2}}{(q)_n (q)_{K-n}} = \frac{(t)_K}{(q)_{K}}, 
\end{equation}

\begin{equation} \label{ii5}
\sum_{n \geq 0} \frac{(a)_n (-1)^n q^{n(n-1)/2} (c/a)^n}{(q)_n (c)_n} = \frac{(c/a)_{\infty}}{(c)_{\infty}} 
\end{equation}
and
\begin{equation} \label{ii29}
\sum_{n \in \mathbb{Z}} \frac{(a)_n z^n}{(b)_n} = \frac{(q, b/a, az, q/az)_{\infty}}{(b, q/a, z, b/az)_{\infty}}
\end{equation}
where
\begin{equation*}
(a_1, \ldots, a_m)_{n} := (a_1)_{n} \cdots (a_m)_{n}.
\end{equation*}
Moreover, we use the convention that
\begin{equation*} 
\frac{1}{(q)_{n}} = 0
\end{equation*}
\noindent for $n<0$. In addition, one can easily check \cite[I.12]{gr}
\begin{equation} \label{neg}
(q^{-n})_k = (-1)^k q^{k(k-1)/2 - nk} \frac{(q)_n}{(q)_{n-k}}.
\end{equation}
We next require the following result. 

\begin{lemma} \label{lm:newqid} For $a$, $b \in \mathbb{Z}_{\geq 0}$, we have
 \begin{equation} \label{eq:newqid}
 (q)_\infty \sum_{l=0}^a (-1)^l \frac{q^{\frac 1 2 l(l+1)}}{(q)_l(q)_{a-l}} \sum_{k=0}^b (-1)^k \frac{q^{\frac 1 2 k(k+1)}}{(q)_k(q)_{b-k}}\,
 \sum_{v\in\Z} q^v \frac{(q)_{v+l+k}}{(q)_{v+l}(q)_{v+k}}
 = \qbinom{a+b}{b}.
\end{equation}
\end{lemma}

\begin{proof}
Note that (\ref{eq:newqid}) follows from the identity
\begin{equation}
\label{eq:toprove1}
(q)_\infty \sum_{k=0}^b (-1)^k \frac{q^{\frac 1 2 k(k+1)}}{(q)_k(q)_{b-k}} \sum_{v\in\Z} q^v \frac{(q)_{v+l+k}}{(q)_{v+l}(q)_{v+k}}
= q^{bl},
\end{equation}
the $t=q^{b+1}$ case of (\ref{ii4}) and simplification. To verify (\ref{eq:toprove1}), we perform the shift $v \to v-k$ to its left-hand side and then 
apply (\ref{ii2}) with $z=q^{v-k+l+1}$. This yields 
\begin{equation} \label{next}
\sum_{k=0}^b (-1)^k \frac{q^{\frac 1 2 k(k+1)}}{(q)_k(q)_{b-k}} \sum_{v\geq0} q^{v-k} \frac{(q)_{v+l}}{(q)_{v}} \sum_{s\geq0} (-1)^s \frac{q^{\frac 1 2 s(s+1)}}{(q)_s} q^{s(v-k+l)}.
\end{equation}  
We now simplify the $k$-sum in (\ref{next}) using the $t=q^{-s}$ case of (\ref{ii4}) and (\ref{neg}). The resulting $v$-sum can then be evaluated by invoking $a=q^{l+1}$, $b=q$ and $z=q^{s+1}$ in (\ref{ii29}) and simplifying. Finally, we perform the shift $s \to s+b$ in the remaining $s$-sum to reveal that (\ref{next}) equals
\begin{equation} \label{ssum}
q^{bl} \sum_{s \geq 0} \frac{(-1)^s q^{s(s+1)/2 + sl} (q)_{s+b}}{(q)_s (q)_{s+b+l+1}}.
\end{equation}
Using $(q)_{s+b} = (q^{b+1})_s (q)_b$ and $(q)_{s+b+l+1} = (q)_{b+l+1} (q^{b+l+2})_{\infty}$, the $s$-sum in (\ref{ssum})
becomes
\begin{equation} \label{ssum1}
\frac{(q)_b}{(q)_{b+l+1}} \sum_{s \geq 0} \frac{(-1)^s q^{s(s+1)/2 + sl} (q^{b+1})_s}{(q)_s (q^{b+l+2})_s}.
\end{equation}
By (\ref{ii5}) with $a=q^{b+1}$ and $c=q^{b+l+2}$ and simplifying, we see that (\ref{ssum1}) equals 1. Thus, (\ref{eq:toprove1}) follows.
\end{proof}

We now proceed as follows. By the diagram for $8_{18}$ in Figure \ref{fig:818} and the algorithm from \cite{glz}, we have
\begin{equation} \label{tail818}
\begin{aligned}
\Phi_{8_{18}}(q)
\=&(q)_\infty^8
\sum_{\substack{v_2,v_3,v_4\geq0\\v_5\in\Z}}
\frac{q^{
v_2v_3+
v_3v_4+
v_4v_5+
v_2v_5+
v_3v_4+
v_2+
v_3+
v_4+
v_5
}
}{
(q)_{v_2}
(q)_{v_3}
(q)_{v_4}
}\\
&\qquad\quad\times\sum_{f_1\geq0} 
(-1)^{f_1}
\frac{
q^{\frac 3 2 f_1(f_1+1)-f_1
+f_1(v_2+v_5)}
}{
(q)_{f_1}
(q)_{f_1+v_2}
(q)_{f_1+v_5}
}
\\
&\qquad\quad\times\sum_{f_2\in\Z}
(-1)^{f_2}
\frac{
q^{\frac 3 2 f_2(f_2+1)-f_2
+f_2(v_2+v_3+v_5)}
}{
(q)_{f_2+v_2}
(q)_{f_2+v_3}
(q)_{f_2+v_5}
}
\\
&\qquad\quad\times\sum_{f_3\in\Z} 
(-1)^{f_3}
\frac{
q^{\frac 3 2 f_3(f_3+1)-f_3
+f_3(v_3+v_4+v_5)}
}{
(q)_{f_3+v_3}
(q)_{f_3+v_4}
(q)_{f_3+v_5}
}
\\
&\qquad\quad\times\sum_{f_4\in\Z} 
(-1)^{f_4}
\frac{
q^{\frac 3 2 f_4(f_4+1)-f_4
+f_4(v_4+v_5)}
}{
(q)_{f_4}
(q)_{f_4+v_4}
(q)_{f_4+v_5}
}.
\end{aligned}
\end{equation}
Using \cite[Lemma 2.1]{ko}, the $f_1$-sum and $f_4$-sum in (\ref{tail818}) can be simplified to
\begin{equation} \label{f1}
\frac{1}{(q)_{\infty}} \sum_{i \geq 0} (-1)^i \frac{q^{\frac{1}{2}i(i+1) + iv_2}}{(q)_i (q)_{i+v_5}}
\end{equation}
and
\begin{equation} \label{f4}
\frac{1}{(q)_{\infty}} \sum_{j \geq 0} (-1)^j \frac{q^{\frac{1}{2} j(j+1) + jv_4}}{(q)_j (q)_{j+v_5}},
\end{equation}
respectively. We now make the substitution $f_2 \to f_2 - v_2$ and then apply \cite[Lemma 2.1]{ko} to the $f_2$-sum in (\ref{tail818}) which yields
\begin{equation} \label{f2}
(-1)^{v_2} \frac{q^{\frac{1}{2} v_2(v_2 - 1) - v_2(v_3 + v_5)}}{(q)_{\infty}} \sum_{k \geq 0} (-1)^k \frac{q^{\frac{1}{2}k(k+1) + k(v_3 - v_2)}}{(q)_k (q)_{k+v_5 - v_2}}.
\end{equation}
Similarly, shifting $f_3 \to f_3 - v_4$ and then applying \cite[Lemma 2.1]{ko} to the $f_3$-sum in (\ref{tail818}) leads to
\begin{equation} \label{f3}
(-1)^{v_4} \frac{q^{\frac{1}{2} v_4 (v_4 - 1) - v_4(v_3 + v_5)}}{(q)_{\infty}} \sum_{l \geq 0} (-1)^l \frac{q^{\frac{1}{2}l(l+1) + l(v_3 - v_4)}}{(q)_l (q)_{l + v_5 - v_4}}.
\end{equation}
Inserting (\ref{f1})--(\ref{f3}) into (\ref{tail818}) gives
\begin{equation} \label{newtail818}
\begin{aligned}
\Phi_{8_{18}}(q) \= (q)_\infty^4 &\sum_{\substack{v_2,v_3,v_4\geq0\\v_5\in\Z}} (-1)^{v_2+v_4}
\frac{
q^{\frac 1 2 v_2(v_2+1)
+\frac 1 2 v_4(v_4+1)
+v_3v_5+v_3+v_5}}
{(q)_{v_2}(q)_{v_3}(q)_{v_4}}\\
&\quad\times\sum_{i\geq0}(-1)^i  \frac{q^{\frac 1 2 i(i+1)+iv_2}} {(q)_i(q)_{i+v_5}}
\quad\sum_{j\geq0}(-1)^j
\frac{q^{\frac 1 2 j(j+1)+jv_4}}{(q)_j(q)_{j+v_5}}
\\
&\quad\times\sum_{l\geq0}(-1)^l \frac{q^{\frac 1 2 l(l+1)+l(v_3-v_4)}}{(q)_l(q)_{l+v_5-v_4}}
\quad\sum_{k\geq0}(-1)^k
\frac{q^{\frac 1 2 k(k+1)+k(v_3-v_2)}}{(q)_k(q)_{k+v_5-v_2}}.
\end{aligned}
\end{equation}
We now simplify the $v_2$-sum, $v_4$-sum and $v_3$-sum in (\ref{newtail818}) by taking $t=q^{1+i-k}$ in (\ref{ii4}), $t=q^{1+j-l}$ in (\ref{ii4}) and $t=q^{v_5+k+l+1}$ in (\ref{ii1}), respectively. Hence, we have
\begin{equation} \label{laststep} 
\begin{aligned}
\Phi_{8_{18}}(q) \=
(q)_\infty^3
&\sum_{i\geq0}(-1)^i 
\frac
{q^{\frac 1 2 i(i+1)}}
{(q)_i}
\quad\sum_{j\geq0}(-1)^j
\frac
{q^{\frac 1 2 j(j+1)}}
{(q)_j}
\\
&\quad\times\sum_{l=0}^j(-1)^l
\frac
{q^{\frac 1 2 l(l+1)}}
{(q)_l(q)_{j-l}}
\quad\sum_{k=0}^i(-1)^k
\frac
{q^{\frac 1 2 k(k+1)}}
{(q)_k(q)_{i-k}}
\sum_{v_5\in\Z}q^{v_5}
\frac {(q)_{v_5+k+l}}{(q)_{k+v_5}(q)_{l+v_5}}.
\end{aligned}
\end{equation}
Equation \eqref{eq:tail818} now follows from Lemma \ref{lm:newqid} and (\ref{laststep}). 

\section*{Acknowledgements} 
The first author is grateful to the Max-Planck-Institut f{\"u}r Mathematik for their hospitality and support as this work began during his stay from May 1-31, 2023. The second author would like to thank his Ph.D. advisor Don Zagier for enlightening discussions. Part of this work appeared in \cite{thesis}. The authors were partially funded by the Irish Research Council Advanced Laureate Award IRCLA/2023/1934. The second author was also partially supported by the Max-Planck-Gesellschaft. Finally, the authors thank the referees for their careful remarks and suggestions which improved the paper.

\section*{Data availability statement}
We do not analyze or generate any datasets, because our work proceeds within a theoretical and mathematical approach.

\section*{Conflict of interest statement}
On behalf of all authors, the corresponding author states that there is no conflict of interest.

\end{document}